\documentclass[12pt, a4paper, oneside, reqno]{amsart} 
\usepackage{cmap}
\usepackage[utf8]{inputenc}
\usepackage[english]{babel}
\usepackage{amssymb,amsfonts,amsmath,cite,enumerate,float,indentfirst, amsthm, mathrsfs}
\usepackage{chngcntr}

\newtheorem{theorem}{Theorem}[section]
\newtheorem{lemma}[theorem]{Lemma}
\newtheorem{proposition}[theorem]{Proposition}
\newtheorem{corollary}[theorem]{Corollary}

\theoremstyle{definition}
\newtheorem{definition}{Definition}[section]
\theoremstyle{definition}
\newtheorem{construction}{Construction}[section]

\theoremstyle{definition}
\newtheorem{remark}{Remark}[section]

\usepackage{microtype}
\usepackage{hyperref}
\usepackage{mathtext}
\usepackage{graphicx}
\usepackage{tikz}
\usepackage{wrapfig}
\usepackage{verbatim}
\usepackage{pgfplots}
\usepackage{subcaption}
\usepackage{url}
\tikzstyle{red dot}=[fill=red, draw=black, shape=circle]
\tikzstyle{green dot}=[fill=green, draw=white, shape=rectangle]
\tikzstyle{new style 2}=[ultra thick]
\tikzstyle{new style 1}=[dashed,-]
\tikzstyle{new style 3}=[- , draw=ultra thick]

\pgfplotsset{compat=1.9}
\makeatletter
\numberwithin{equation}{section}

\newcommand{\Mod}{\mathrm{Mod}}
\newcommand{\Homeo}{\mathrm{Homeo}}
\hypersetup{
 colorlinks=true,
 linkcolor=red,
 filecolor=magenta, 
 urlcolor=cyan,
 pdftitle={Overleaf Example},
 }
 
\begin{document}
       
\title[Simple Closed Geodesics on Hyperbolic Surfaces]{Lengths of Simple Closed Geodesics on Hyperbolic Surfaces in Prescribed Homology Classes}
\author[Igor Patsankov]{Igor M. Patsankov} 
\address{Lomonosov Moscow State University, Russia}
\address{Steklov Mathematical Institute of Russian Academy of Sciences, Moscow, Russia}
\email{ipatsankov@mail.ru}
\thanks{This work was supported by the Russian Science Foundation under grant no.~26-11-00052, \\
\url{https://rscf.ru/en/project/26-11-00052/}.}
\keywords {Hyperbolic surface, closed geodesics, mapping class group, cycle complex.}
\subjclass[2020]{57K20, 53C22, 51M10}
\date{\today}
\begin{abstract}
    A classical question in the theory of hyperbolic surfaces is the study of lengths of closed 
    geodesics under various constraints. A celebrated result in this area is 
    M. Mirzakhani's asymptotic formula for the number of simple closed geodesics of length $\le L$ 
    on a hyperbolic surface of genus $g$ with $n$ punctures. We investigate the number of simple 
    closed geodesics of length $\le L$ representing a fixed primitive nonzero homology class $x$ on a 
    hyperbolic surface $S$. We denote this number by $h_{S}(L, x)$. It follows from Mirzakhani's result 
    that $h_{S}(L, x) \le C L^{6(g-1) + 2n}$. However, numerical evidence suggests 
    that this bound is apparently not asymptotically sharp. We prove that for a surface $S$ of 
    genus $g$ with $n$ punctures and $b$ geodesic boundary components, under the condition that 
    $g \ge 1$ and $g+n+b \ge 3$, there exists a constant $C_1 > 0$ such that for sufficiently 
    large $L$ the inequality
    \[ h_{S}(L, x) \ge C_1 L^{6(g-1) + 2(n + b-1)} \]
    holds. In the special case of a torus with two punctures $S_{1, 2}$, we 
    obtain the following stronger result: there exists a constant $C_2 > 0$ such that for 
    sufficiently large $L$ the inequality
    \[ h_{S_{1, 2}}(L, x) \ge C_2 L^{3.011057 \ldots  } \]
    holds.
\end{abstract}

\maketitle

 \section{Introduction} 

    Let $h_{S}(L, x)$ be the number of oriented simple closed geodesics of length $\le L$
    on a hyperbolic surface $S$ which belong to a primitive nonzero homology class
    $x \in H_1(S, \mathbb{Z})$.
    Throughout this paper we assume that $S$ is an oriented finite-type surface
    with a complete finite-area metric and geodesic boundary components.
    The main goal of our work is to obtain lower bounds for the function $h_S(L,x)$.

    The problem of understanding the asymptotics of lengths of closed geodesics under various 
    constraints has a rich history.
    Thanks to the works of Delsarte, Huber and Selberg, it is known that  
    \[
        c_{S}(L) \sim e^L/L, 
    \]
    where $c_{S}(L)$ denotes the number of primitive closed geodesics of length
     $\le L$ on a hyperbolic surface $S$.
    For more details on this result, see \cite{Bus}.
    Analogous statements hold for the growth rate of the number of closed geodesics on compact manifolds 
    of negative curvature \cite{Ma}.

    Of particular interest are simple closed geodesics.
    In 2008, Mirzakhani published a remarkable paper \cite{Mirz} devoted to counting simple closed geodesics,
    in which she obtained an asymptotic formula for the number of simple closed geodesics of length at most 
    $L$ on hyperbolic surfaces.  In contrast to the previous result, Mirzakhani's formula depends on
    the topological type of the surface. 
    
    Let $S_{g, n}$ denote a hyperbolic surface of genus $g$ with $n$
    punctures, and let $s_{g, n}(L)$ be the number of simple closed geodesics of length $\le L$ on $S_{g, n}$.
    Then the main result of \cite{Mirz} states that
    \[
        s_{g, n}(L) \sim c L^{6g-6+2n}, 
    \]
    where the constant $c$ depends only on the hyperbolic metric and the topological type of the surface.

    Alongside the general problem of counting simple closed geodesics, Mirzakhani also solved the problem of 
    counting geodesics in individual orbits of the mapping class group $\mathrm{Mod}_{g, n}$.
     Consider the counting function
    \[
        s_{g, n}(L, \gamma) = \# \{\alpha \in \mathrm{Mod}_{g, n} \cdot \gamma \mid \ell_{\alpha} \le L \},
    \]  
    where $\gamma$ is a fixed simple closed curve on $S_{g, n}$, and $\ell_{\alpha}$ denotes
     the length of the geodesic representative in the free isotopy class of $\alpha$. 
     Since there are only finitely many mapping class group orbits of simple closed curves, 
     the counting function $s_{g, n}(L)$ decomposes as the finite sum
    \[
        s_{g, n}(L) = \sum_{\gamma} s_{g, n}(L, \gamma),
    \]
    where the sum runs over a set of representatives of the orbits. Mirzakhani proved that for any 
    simple closed curve $\gamma$, one has the asymptotic
    \begin{equation} \label{eq_Mirz}
        s_{g, n}(L, \gamma)\sim c_{\gamma} L^{6g-6+2n},
    \end{equation}
    where the constant $c_{\gamma}$ now depends on the topological type of the curve, as well as on the hyperbolic metric.

    A natural problem now arises: given a primitive homology class $x \neq 0$,
    what is the growth rate of the number of simple closed geodesics of length $\le L$ representing $x$?
    \begin{remark}
        In the context of the above-mentioned results of Delsarte, Huber, and Selberg 
        on not necessarily simple geodesics, an analogous problem was considered by 
        Phillips and Sarnak \cite{P_S}.
        As far as the author knows, for simple closed geodesics this question has 
        not been studied so far.
    \end{remark}
    It is well known that on closed surfaces, the homology classes representable by simple 
    closed non-separating curves are precisely the primitive ones, 
    i.e. those not divisible by any integer greater than $1$. 
    For non-closed surfaces, we need the following definition.

   \begin{definition}
    A homology class $x \in H_1(S, \mathbb{Z})$ is called \textit{strictly primitive} if its image 
    is primitive after capping off all boundary components and punctures with disks.
    \end{definition}
    It is easy to see that for non-closed surfaces, the homology classes representable by simple
    closed non-separating geodesics are exactly the strictly primitive classes. 
    For closed surfaces, the notions of primitive and strictly primitive coincide.

    Let $S_{g, n}^{b}$ denote a surface of genus $g$ with $n$ punctures and $b$ boundary components.
    Let $x \in H_{1}\left(S_{g, n}^{b},\mathbb{Z}\right)$ be a nonzero strictly primitive
    homology class.
    Let $X$ be a complete finite-area hyperbolic metric on $S_{g, n}^{b}$ such
    that the boundary components are geodesic.
    Using the above notation, we define the counting function
    \[
        h_{g,n}^{b}(L, x, X) = \# \{ \alpha \in x \mid  \ell_{\alpha}(X) \le L\}, 
    \] 
    where $\alpha \in x$ means that $\alpha$ is an oriented simple closed geodesic
    representing the class $x$.
    \begin{remark}
    In what follows, by a hyperbolic metric on $S_{g, n}^{b}$ we shall always mean a complete
    finite-area hyperbolic metric for which
    the boundary components are geodesic. Such a metric exists if
    and only if the Euler characteristic $\chi (S_{g, n}^{b}) < 0$.
    \end{remark}

    A natural first question is  whether the leading order of growth of the function $h$ depends on 
    the choice of metric $X$, 
    the strictly primitive nonzero homology class $x$, or the replacement of a geodesic boundary component by a puncture. 
    The answer to this question is given by our first result.
    \begin{theorem} \label{th_nezavisimost_description}
        Let $S_{g_1, n_1}^{b_1}$ and $S_{g_2, n_2}^{b_2}$ be surfaces such that 
        $b_1 + n_1 = b_2 + n_2$, let $X_1, X_2$ be hyperbolic metrics on 
        $S_{g_1, n_1}^{b_1}$, $S_{g_2, n_2}^{b_2}$, respectively,
        and let $x_1 \in H_{1}\left(S_{g_1, n_1}^{b_1}, \mathbb{Z}\right)$, $x_2 \in H_{1}\left(S_{g_2, n_2}^{b_2}, \mathbb{Z}\right)$ be strictly primitive homology classes.
        Then
        there exists a constant $c > 0$ such that for all sufficiently large $L$ the inequality
        \[ h_{g_2, n_2}^{b_2}\left(\frac{L}{c}, x_2,X_2\right) \le h_{g_1, n_1}^{b_1}\Big(L, x_1,X_1\Big) \le h_{g_2, n_2}^{b_2} \Big(c L, x_2,X_2 \Big) \]
        holds.
    \end{theorem}
    This theorem will be proved in Section \ref{RefProject_General_case}.
    \begin{remark}
        In view of this theorem, we reduce the general problem to the case $h_{g,n}(L, x)$, 
        where the metric symbol is omitted for simplicity of notation.
    \end{remark}

    Consider the function $s_{g, n}(L, \gamma_0)$, where $\gamma_0$ belongs to the topological
    class of non-separating curves.
    It is well known that the homology classes represented by simple closed non-separating
    geodesics coincide precisely with the nonzero strictly primitive classes.
    Thus, the function $s_{g, n}(L, \gamma_0)$ can be written as an infinite
    sum over strictly primitive classes $x$ of the functions $h_{g,n}(L, x)$:
    \[
        s_{g, n}(L, \gamma_0) = \sum_{x - \text{ strictly primitive}} h_{g, n}(L, x).
    \]       
    As we have already noted, we are interested in the following question:
    does the growth order of the functions $h_{g, n}(L, x)$ coincide with that of $s_{g, n}(L, \gamma_0)$?
    
    The case $g=0$ is uninteresting, since on surfaces of genus $0$ there are no strictly primitive 
    homology classes, and 
    the homology classes of simple closed geodesics 
    are in bijection with the topological types of 
    simple closed geodesics, up to orientation.
    Note that in the simplest case of the once-punctured torus, for each strictly primitive 
    homology class there is a unique simple closed geodesic.
    Therefore, for sufficiently large $L$, the function $h_{1,1}(L,x) = 1$. Thus, in this case
    the growth order is strictly smaller than the growth order $cL^2$ from Mirzakhani's estimate.

    We have obtained the following general estimate.  
    \begin{theorem} \label{th_general_description}
        Let $X$ be a fixed hyperbolic metric on the surface $S_{g,n}$, with $g \ge 1$ and $g+n \ge 3$. Then for
        any strictly primitive homology class $x \neq 0$ there exist constants $C_1, C_2, L_0 > 0$ such that
        for any $L \ge L_0$,
    \[ C_1 L^{6(g-1) + 2(n-1)} \le h_{g,n}(L,x) \le C_2 L^{6(g-1) + 2n}.\]
    \end{theorem}
    In this theorem, the lower bound is the essential one, while the upper bound follows from Mirzakhani's theorem.
        
    Our main result is a lower bound for $h$ in the first nontrivial 
    case, namely a torus with two punctures. In this case, Mirzakhani's upper bound 
    is of the form $c L^4$, while our general lower bound from Theorem \ref{th_general_description}
    is of the form $c L^2$. We shall prove the following theorem.
    \begin{theorem} \label{th_S_1_2_description}
    Let $X$ be a fixed hyperbolic metric on the surface $S_{1,2}$.
    Then for any fixed 
    strictly primitive nonzero homology class $x \in H_1(S_{1,2}, \mathbb{Z})$,
    there exist constants $C > 0$ and $L_0 > 0$ such that
    for any $L \ge L_0$,
    \[
        h_{1,2}(L, x) \ge C L^{3 + \eta},
    \]
    where $\eta = 0.011057 \ldots$
   \end{theorem}  

    \begin{remark}
        The weaker bound $h_{1,2} \ge C L^3$ is significantly easier to obtain. 
        At an early stage of this work, the estimate $C L^3$ appeared to be a plausible 
        candidate for the correct asymptotics. We therefore consider it especially 
        significant that we have been able to prove a bound strictly stronger than $3$.
    \end{remark}
    The constants in the estimates from Theorems \ref{th_general_description} and \ref{th_S_1_2_description}
    depend on the metric and on the class $x$. Despite their different appearances, 
    all these estimates are united by a common idea.
    We will give a sketch of the proof using Figure \ref{fig_00}.
    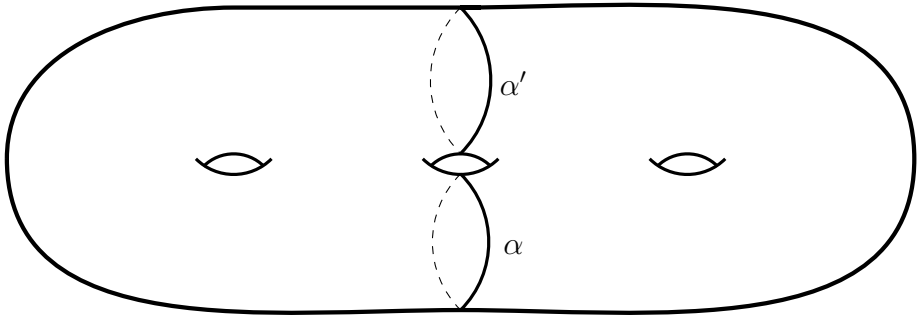
\begin{figure}[h]
        \centering
        \begin{tikzpicture} 
            \draw (7.7,-1.4) -- node[above] {$\alpha$} cycle ; 
            \draw (7.7,0.7) -- node[above] {$\alpha'$} cycle ; 
            \draw[ultra thick] (1,0) to [out=90,in=180] (4,2) to [out=0,in=0] (7,2) to
            [out=0,in=90] (13,0) to [out=-90,in=0] (7,-2) to [out=180,in=-90] (1,0); 
            \draw[very thick] (7,2) to [out=-45,in=45] (7,0.07);
            \draw[dashed] (7,2) to [out=-135,in=135] (7,0.07);
            \draw[very thick] (3.5,0) to [out=-45,in=-135] (4.5, 0); 
            \draw[very thick] (3.6,-0.1) to [out=45,in=135] (4.4, -0.1); 
            \draw[very thick] (6.5,0) to [out=-45,in=-135] (7.5, 0); 
            \draw[very thick] (6.6,-0.1) to [out=45,in=135] (7.4, -0.1); 
            \draw[very thick] (9.5,0) to [out=-45,in=-135] (10.5, 0); 
            \draw[very thick] (9.6,-0.1) to [out=45,in=135] (10.4, -0.1); 
            \draw[very thick] (7,-0.2) to [out=-45,in=45] (7,-2);
            \draw[dashed] (7,-0.2) to [out=-135,in=135] (7,-2);
        \end{tikzpicture}
        \caption{ Bounding pair.}
        \label{fig_00}

    \end{figure}

    Let $[\alpha] = x$. 
    Note that $[\alpha'] = [\alpha]$, since these geodesics together separate the surface. 
    Consider the auxiliary surface $S_{2}^{2}$, obtained from $S_{3,0}$ by cutting along $\alpha$. 
    Then, using the corollary from Mirzakhani's theorem (see Theorem \ref{th_general_Mirz}), 
    we obtain that for sufficiently large $L$ the following inequalities hold:
    \[\frac{1}{C} L^{10} \le s_{2}^{2}(L, \alpha') \le  C L^{10}.\]
    Moreover, for any geodesic $\gamma \in \mathrm{Mod}_{2}^{2} \cdot \alpha'$, we have $[\gamma] = [\alpha]$.
    Consequently, for sufficiently large $L$, we obtain the estimate
    \[
        \frac{1}{C} L^{10}  \le  h_{3,0}(L,x).
    \]
    The lower bounds in Theorem \ref{th_general_description} are obtained  similarly.
   
    The main effort in our work was directed towards obtaining the lower bound for 
    the function $h_{1,2}(L,x)$ from Theorem \ref{th_S_1_2_description}. To prove this, 
    we apply the above construction iteratively, at each step cutting along the curves 
    produced earlier, which yields an improved estimate. A key role in describing these 
    iterations is played by the cycle complex $\mathcal{B}_{g}$ and 
    its generalization $\mathcal{B}_{g, n}$.
    This complex was first considered in the paper \cite{Bestvina}, which is devoted to 
    the cohomological dimension of the Torelli group, and was constructed only for closed surfaces.
    In the present work, we consider its direct generalization to surfaces with boundary.
   
    We now outline the proof in this case. We shall identify geodesics representing strictly primitive 
    classes $x + my$, where $m \in \mathbb{Z}$, with vertices of the cycle complex $\mathcal{B}_{1,2}$, 
    where $y$ denotes the class of a curve homotopic to a loop around a puncture. Let $\gamma_0$ denote
    the vertex corresponding to the shortest such geodesic.
    The complex $\mathcal{B}_{1,2}$ is a tree with vertices of infinite degree. Two vertices are joined
     by an edge if the corresponding geodesics are disjoint. Thus, in this case $\mathcal{B}_{1,2}$ is 
    a subcomplex of the curve complex. We prove that the restriction of $\mathcal{B}_{1,2}$ to geodesics 
    of length $\le L$ is a contractible complex. This fact allows us to iteratively estimate the lengths 
    of geodesics corresponding to vertices at a fixed distance in the cycle complex from the distinguished 
    vertex $\gamma_0$, which represents a geodesic of shortest length in the class $x$. The above argument 
    is described in more detail in the proof of Theorem \ref{th_ocenka_S_1_2}.
    After this procedure, we reduce the problem to a noncommutative analogue of the generalized Dirichlet 
    divisor problem, which is a classical number-theoretic problem. The solution of this purely 
    combinatorial problem is the subject of Theorem \ref{th_reshenie_combinatornoy_zadachi}.

    It proved useful for our purposes to consider the surface $S_{0,4}$.
    This case is of interest because for it there is a known necessary and sufficient 
    condition for a closed geodesic to be simple in terms of Dehn--Thurston coordinates,
    which in general describe integral multicurves
    (collections of pairwise disjoint geodesics with integral weights).
    This fact was used in the work \cite{Rivin}.

    We obtain a lower bound for the number of curves representing the class $[\gamma_0]$ and
    corresponding to vertices
    at distance $2k$ in the cycle complex from the distinguished vertex $\gamma_0$. 
    This bound is of the form
    \[
        h_{1,2}(L,x) \ge C_k L^{3}(\log L)^{k-1},
    \]
    where the inequality holds for sufficiently large $L$. 
    We then sum these estimates over all $k$ and obtain
    the lower bound for $h_{1,2}(L,x)$ from Theorem \ref{th_S_1_2_description}. 
    \begin{figure}[H]
        \centering
        \begin{tikzpicture} [scale=0.9]
            \draw (0,0) -- node[below] {$\gamma_0$} cycle ;  
            \filldraw[black] (0,0) circle (1pt);
            \filldraw[black] (0,6.2) circle (0.5pt); 
            \filldraw[black] (0,6.7) circle (0.5pt);
            \filldraw[black] (0,7.2) circle (0.5pt);

            \filldraw[black] (6.2, 6.2) circle (0.5pt); 
            \filldraw[black] (6.5, 6.7) circle (0.5pt);
            \filldraw[black] (6.8, 7.2) circle (0.5pt);

            \draw[thick] (-3.5,5.25) to   (-3.75,5.625);
            \draw[thick] (-3.5,5.25) to   (-3.25,5.625);
            \filldraw[black] (-3.75,5.625) circle (1pt);
            \filldraw[black] (-3.25,5.625) circle (1pt);
            \draw[thick] (-2.5,5.25) to   (-2.75,5.625);
            \draw[thick] (-2.5,5.25) to   (-2.25,5.625);
            \filldraw[black] (-2.75,5.625) circle (1pt);
            \filldraw[black] (-2.25,5.625) circle (1pt);
            \draw[thick] (-3,4.5) to   (-3.5,5.25);
            \draw[thick] (-3,4.5) to   (-2.5,5.25);
            \filldraw[black] (-3.5,5.25) circle (1pt);
            \filldraw[black] (-2.5,5.25) circle (1pt);
            \draw[thick] (-1.5,5.25) to   (-1.75,5.625);
            \draw[thick] (-1.5,5.25) to   (-1.25,5.625);
            \filldraw[black] (-1.75,5.625) circle (1pt);
            \filldraw[black] (-1.25,5.625) circle (1pt);
            \draw[thick] (-0.5,5.25) to   (-0.75,5.625);
            \draw[thick] (-0.5,5.25) to   (-0.25,5.625);
            \filldraw[black] (-0.75,5.625) circle (1pt);
            \filldraw[black] (-0.25,5.625) circle (1pt);
            \draw[thick] (-1,4.5) to   (-1.5,5.25);
            \draw[thick] (-1,4.5) to   (-0.5,5.25);
            \filldraw[black] (-1.5,5.25) circle (1pt);
            \filldraw[black] (-0.5,5.25) circle (1pt);
            \draw[thick] (-2,3) to   (-3,4.5);
            \draw[thick] (-2,3) to   (-1,4.5);
            \filldraw[black] (-3,4.5) circle (1pt);
            \filldraw[black] (-1,4.5) circle (1pt);
            \draw[thick] (0.5,5.25) to   (0.25,5.625);
            \draw[thick] (0.5,5.25) to   (0.75,5.625);
            \filldraw[black] (0.25,5.625) circle (1pt);
            \filldraw[black] (0.75,5.625) circle (1pt);
            \draw[thick] (1.5,5.25) to   (1.25,5.625);
            \draw[thick] (1.5,5.25) to   (1.75,5.625);
            \filldraw[black] (1.25,5.625) circle (1pt);
            \filldraw[black] (1.75,5.625) circle (1pt);
            \draw[thick] (1,4.5) to   (0.5,5.25);
            \draw[thick] (1,4.5) to   (1.5,5.25);
            \filldraw[black] (0.5,5.25) circle (1pt);
            \filldraw[black] (1.5,5.25) circle (1pt);
            \draw[thick] (2.5,5.25) to   (2.25,5.625);
            \draw[thick] (2.5,5.25) to   (2.75,5.625);
            \filldraw[black] (2.25,5.625) circle (1pt);
            \filldraw[black] (2.75,5.625) circle (1pt);
            \draw[thick] (3.5,5.25) to   (3.25,5.625);
            \draw[thick] (3.5,5.25) to   (3.75,5.625);
            \filldraw[black] (3.25,5.625) circle (1pt);
            \filldraw[black] (3.75,5.625) circle (1pt);
            \draw[thick] (3,4.5) to   (2.5,5.25);
            \draw[thick] (3,4.5) to   (3.5,5.25);
            \filldraw[black] (2.5,5.25) circle (1pt);
            \filldraw[black] (3.5,5.25) circle (1pt);
            \draw[thick] (2,3) to   (1,4.5);
            \draw[thick] (2,3) to   (3,4.5);
            \filldraw[black] (1,4.5) circle (1pt);
            \filldraw[black] (3,4.5) circle (1pt);
            \draw[thick] (0,0) to   (-2,3);
            \draw[thick] (0,0) to   (2,3);
            \filldraw[black] (-2,3) circle (1pt);
            \filldraw[black] (2,3) circle (1pt);

            \draw[thick, dotted] (-3.5,4.5) to [out=90,in=180] (-3, 4.75) to [out=0,in=180] (3,4.75) to [out=0,in=90] (3.5,4.5);
            \draw[thick, dotted] (-3.5,4.5) to [out=-90,in=180] (-3, 4.25) to [out=0,in=180] (3,4.25) to [out=0,in=-90] (3.5,4.5);

            \draw[thick, dotted] (-4.25,5.625) to [out=90,in=180] (-4,5.8) to [out=0,in=180] (4,5.80) to [out=0,in=90] (4.25,5.625);
            \draw[thick, dotted] (-4.25,5.625) to [out=-90,in=180] (-4,5.45) to [out=0,in=180] (4,5.45) to [out=0,in=-90] (4.25,5.625);
                
            \draw (4,4.5) -- node[right] {$\ge C_1 L^3 $} cycle ;
            \draw (4.75,5.625) -- node[right] {$\ge C_2 L^3 \log L $} cycle ;
        \end{tikzpicture} 
        \caption{Iterative bound.}
    \end{figure}
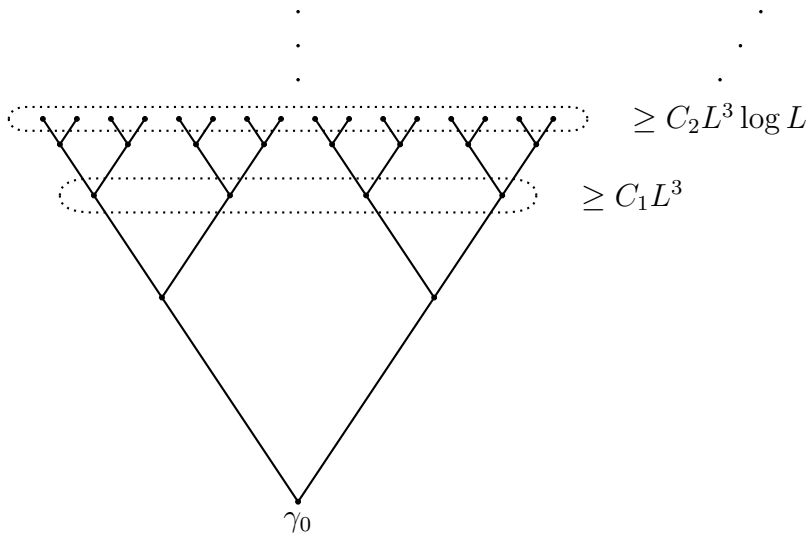
    On the odd layers, the class $[\gamma_0]$ does not arise, since the vertices of the 
    $k$-th layer represent homology classes of the form $[\gamma_0] + qy$, where the parity of 
    $q$ coincides with the parity of $k$.

    In addition, during this work we wrote a program for the numerical estimation of the growth 
    order of the function $h_{1,2}(L, x)$.
    Assume that $h_{1,2}(L, x) = C_x L^{\alpha_x} (1 + o(1))$ as $L \to \infty$. 
    Note that the exponent $\alpha_x$ is approximated by the function 
    $\alpha(L, x) = \log_2 \left( \frac{h_{1,2}(2L, x)}{h_{1,2}(L, x)} \right)$
     as $L$ tends to infinity. Figure \ref{fig0} shows the graph of $\alpha(L, x)$ 
     for both the zero and a nonzero strictly primitive class.
    One clearly sees that $\alpha(L, 0)$ converges to $4$, which agrees with
     Mirzakhani's result for separating curves. This provides evidence that the program 
     correctly estimates the growth order of $h_{1,2}(L, x)$. 
     In the case of a nonzero class $x$, the function $\alpha(L, x)$ appears to 
     tend to a value close to $3.5$. This suggests that the growth order of 
     $h_{1,2}(L, x)$ is strictly smaller than that of $s_{1,2}(L)$, which is $4$.
    \begin{figure}[h]
        \centering
        \begin{tikzpicture}
    \begin{axis}[
            legend pos = north east, 
            height = 0.31\paperheight, 
            width = 0.82\paperwidth,
            xmin = 10,
            xmax = 200,
            line width=1pt
        ]
            
                        \legend{ 
                            $ x = 0 $, 
                            $ x \neq 0$, 
                            $4$, 
                            $3.5 $
                            };

                        \addplot[solid,  draw = violet] table[row sep=crcr]  {
                        17 4.42146 \\
                18 4.06465 \\
                19 4.36768 \\
                20 4.02393 \\
                21 4.2963 \\
                22 4.04782 \\
                23 4.29912 \\
                24 4.00729 \\
                25 4.24021 \\
                26 4.02259 \\
                27 4.23129 \\
                28 3.99053 \\
                29 4.1869 \\
                30 3.98456 \\
                31 4.17322 \\
                32 3.97766 \\
                33 4.14941 \\
                34 3.98959 \\
                35 4.15543 \\
                36 3.98701 \\
                37 4.1435 \\
                38 3.9964 \\
                39 4.13985 \\
                40 3.97536 \\
                41 4.1162 \\
                42 3.97917 \\
                43 4.11452 \\
                44 3.97661 \\
                45 4.10118 \\
                46 3.97824 \\
                47 4.10222 \\
                48 3.97616 \\
                49 4.09504 \\
                50 3.9801 \\
                51 4.08995 \\
                52 3.97417 \\
                53 4.08376 \\
                54 3.97896 \\
                55 4.08469 \\
                56 3.97646 \\
                57 4.07447 \\
                58 3.97955 \\
                59 4.0783 \\
                60 3.98088 \\
                61 4.07676 \\
                62 3.98351 \\
                63 4.0723 \\
                64 3.97349 \\
                65 4.06288 \\
                66 3.97651 \\
                67 4.06344 \\
                68 3.97668 \\
                69 4.05739 \\
                70 3.97545 \\
                71 4.05748 \\
                72 3.97683 \\
                73 4.05704 \\
                74 3.97993 \\
                75 4.05424 \\
                76 3.97619 \\
                77 4.05204 \\
                78 3.98059 \\
                79 4.05479 \\
                80 3.97952 \\
                81 4.04844 \\
                82 3.9793 \\
                83 4.04961 \\
                84 3.98132 \\
                85 4.04975 \\
                86 3.98289 \\
                87 4.04714 \\
                88 3.97834 \\
                89 4.04384 \\
                90 3.98118 \\
                91 4.04552 \\
                92 3.9815 \\
                93 4.04152 \\
                94 3.98005 \\
                95 4.04139 \\
                96 3.98109 \\
                97 4.04136 \\
                98 3.98245 \\
                99 4.03873 \\
                100 3.9796 \\
                101 4.03742 \\
                102 3.98254 \\
                103 4.03942 \\
                104 3.981 \\
                105 4.03382 \\
                106 3.97963 \\
                107 4.03431 \\
                108 3.9825 \\
                109 4.0362 \\
                110 3.98368 \\
                111 4.03395 \\
                112 3.9803 \\
                113 4.0322 \\
                114 3.98336 \\
                115 4.0341 \\
                116 3.98375 \\
                117 4.03153 \\
                118 3.98256 \\
                119 4.03166 \\
                120 3.9839 \\
                121 4.03232 \\
                122 3.98465 \\
                123 4.03003 \\
                124 3.98228 \\
                125 4.02883 \\
                126 3.98428 \\
                127 4.03042 \\
                128 3.98394 \\
                129 4.02718 \\
                130 3.98278 \\
                131 4.0275 \\
                132 3.9847 \\
                133 4.02874 \\
                134 3.98565 \\
                135 4.02676 \\
                136 3.98281 \\
                137 4.02552 \\
                138 3.98544 \\
                139 4.02757 \\
                140 3.98558 \\
                141 4.02516 \\
                142 3.98456 \\
                143 4.02557 \\
                144 3.98596 \\
                145 4.02622 \\
                146 3.98629 \\
                147 4.02419 \\
                148 3.98466 \\
                149 4.02398 \\
                150 3.98675 \\
                151 4.02557 \\
                152 3.98665 \\
                153 4.02311 \\
                154 3.98529 \\
                155 4.02286 \\
                156 3.98683 \\
                157 4.02416 \\
                158 3.98711 \\
                159 4.02221 \\
                160 3.98457 \\
                161 4.02091 \\
                162 3.98647 \\
                163 4.02248 \\
                164 3.98669 \\
                165 4.02031 \\
                166 3.9854 \\
                167 4.02055 \\
                168 3.98689 \\
                169 4.02162 \\
                170 3.98759 \\
                171 4.02023 \\
                172 3.98603 \\
                173 4.01991 \\
                174 3.98781 \\
                175 4.02107 \\
                176 3.98715 \\
                177 4.01868 \\
                178 3.9862 \\
                179 4.01894 \\
                180 3.98797 \\
                181 4.02036 \\
                182 3.9883 \\
                183 4.01879 \\
                184 3.98654 \\
                185 4.01807 \\
                186 3.98811 \\
                187 4.0195 \\
                188 3.98838 \\
                189 4.01785 \\
                190 3.98721 \\
                        };
                
                        \addplot[solid,  draw = teal] table[row sep=crcr]  {
                        17 4.64386 \\
                18 5.08746 \\
                19 4.44294 \\
                20 4.72792 \\
                21 4.44846 \\
                22 4.72792 \\
                23 4.56986 \\
                24 4.79442 \\
                25 4.29068 \\
                26 4.51907 \\
                27 4.22497 \\
                28 4.41954 \\
                29 4.25956 \\
                30 4.47199 \\
                31 4.22386 \\
                32 4.39819 \\
                33 4.06695 \\
                34 4.24336 \\
                35 3.95146 \\
                36 4.10236 \\
                37 3.91461 \\
                38 4.05889 \\
                39 3.91684 \\
                40 4.0518 \\
                41 3.89139 \\
                42 4.03132 \\
                43 3.88802 \\
                44 4.00961 \\
                45 3.87568 \\
                46 3.99382 \\
                47 3.87933 \\
                48 3.99145 \\
                49 3.81074 \\
                50 3.92052 \\
                51 3.8035 \\
                52 3.90779 \\
                53 3.79128 \\
                54 3.89449 \\
                55 3.80253 \\
                56 3.89334 \\
                57 3.77386 \\
                58 3.86913 \\
                59 3.74327 \\
                60 3.83372 \\
                61 3.74637 \\
                62 3.83111 \\
                63 3.74238 \\
                64 3.82636 \\
                65 3.72708 \\
                66 3.81239 \\
                67 3.71791 \\
                68 3.79226 \\
                69 3.72128 \\
                70 3.79723 \\
                71 3.71894 \\
                72 3.79335 \\
                73 3.70298 \\
                74 3.77141 \\
                75 3.6983 \\
                76 3.76802 \\
                77 3.69454 \\
                78 3.76424 \\
                79 3.69785 \\
                80 3.7612 \\
                81 3.68494 \\
                82 3.7512 \\
                83 3.67208 \\
                84 3.73499 \\
                85 3.66401 \\
                86 3.72415 \\
                87 3.66348 \\
                88 3.72323 \\
                89 3.6587 \\
                90 3.71961 \\
                91 3.66044 \\
                92 3.7151 \\
                93 3.66227 \\
                94 3.7193 \\
                95 3.65961 \\
                96 3.71507 \\
                97 3.6495 \\
                98 3.70159 \\
                99 3.64595 \\
                100 3.69878 \\
                101 3.6387 \\
                102 3.69122 \\
                103 3.63894 \\
                104 3.68734 \\
                105 3.63494 \\
                106 3.68486 \\
                107 3.62901 \\
                108 3.67759 \\
                109 3.6253 \\
                110 3.67074 \\
                111 3.62804 \\
                112 3.67454 \\
                113 3.62116 \\
                114 3.66855 \\
                115 3.62038 \\
                116 3.66293 \\
                117 3.62125 \\
                118 3.66574 \\
                119 3.61673 \\
                120 3.66091 \\
                121 3.61323 \\
                122 3.65453 \\
                123 3.61252 \\
                124 3.65486 \\
                125 3.61026 \\
                126 3.65261 \\
                127 3.6104 \\
                128 3.64911 \\
                129 3.6066 \\
                130 3.64758 \\
                131 3.60043 \\
                132 3.64034 \\
                133 3.59869 \\
                134 3.63606 \\
                135 3.60034 \\
                136 3.63897 \\
                137 3.59519 \\
                138 3.63368 \\
                139 3.59575 \\
                140 3.63143 \\
                141 3.59655 \\
                142 3.6332 \\
                143 3.59334 \\
                144 3.6298 \\
                145 3.59006 \\
                146 3.62419 \\
                147 3.59182 \\
                148 3.62684 \\
                149 3.58953 \\
                150 3.6248 \\
                151 3.59004 \\
                152 3.62253 \\
                153 3.58791 \\
                154 3.62187 \\
                155 3.5843 \\
                156 3.61824 \\
                157 3.58361 \\
                158 3.61451 \\
                159 3.58393 \\
                160 3.6168 \\
                161 3.57977 \\
                162 3.61227 \\
                163 3.57838 \\
                164 3.60832 \\
                165 3.57951 \\
                166 3.61103 \\
                167 3.57785 \\
                168 3.60877 \\
                169 3.57535 \\
                170 3.60471 \\
                171 3.57535 \\
                172 3.6054 \\
                173 3.57336 \\
                174 3.60357 \\
                175 3.57389 \\
                176 3.60192 \\
                177 3.5733 \\
                178 3.6028 \\
                179 3.5693 \\
                180 3.59852 \\
                181 3.56929 \\
                182 3.5964 \\
                183 3.57011 \\
                184 3.59839 \\
                185 3.56815 \\
                186 3.59636 \\
                187 3.56741 \\
                188 3.59357 \\
                189 3.56872 \\
                190 3.59628 \\
                        };
                
                        \addplot[solid,  draw = black]  table[row sep=crcr] {
                        a  b \\
                        17 4 \\
                        190 4 \\
                        };
                    
                        \addplot[solid,  draw = orange]  table[row sep=crcr]  {
                        a   c \\
                        17  3.5 \\
                        190  3.5 \\
                        };
                    
                    \end{axis}
        \end{tikzpicture}
            
                \caption{Graphs of the function $\alpha(L, x) = \log_2 \left( \frac{h_{1,2}(2L, x)}{h_{1,2}(L,x)}\right)$ for various $x$.}
                \label{fig0}
    \end{figure}
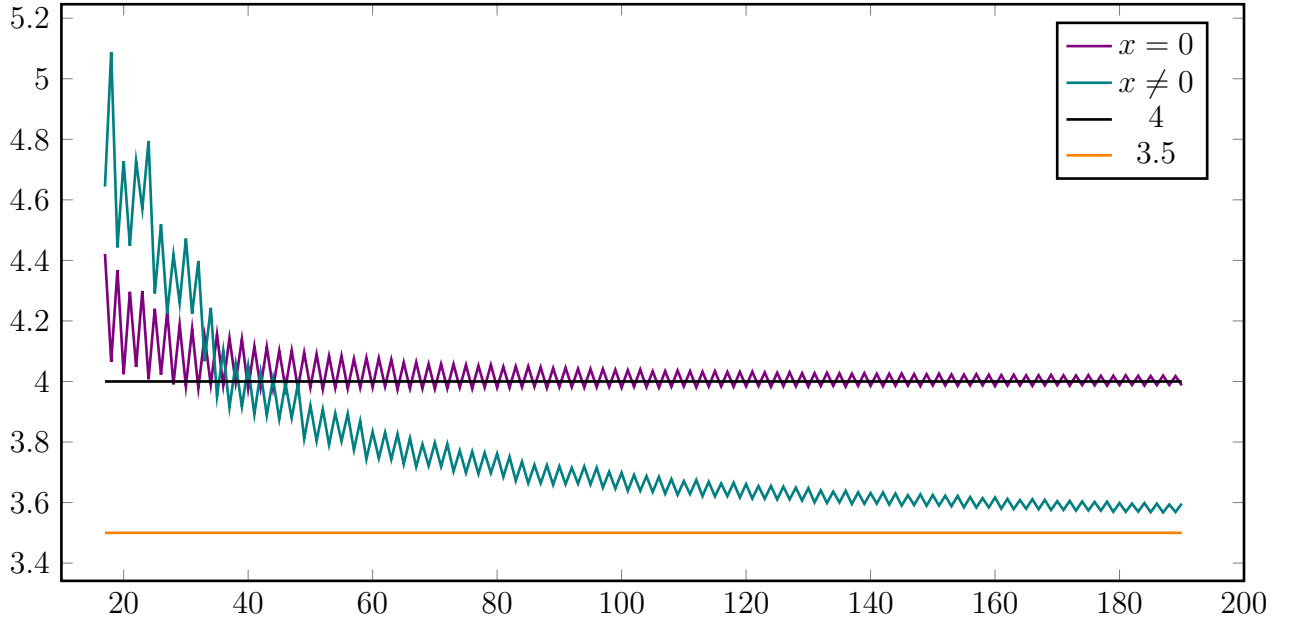

    We shall use the following notation. Let $g, h: \mathbb{R} \rightarrow \mathbb{R}$ be functions.
    Then $g \ll h$ if and only if there exist constants $M, L_0 > 0$ such that for any $L \ge L_0$, we have $|g(L)| \le M h(L)$.
    If both $g \ll h$ and $h \ll g$ hold, we write $g \asymp h$.

    We now briefly outline the structure of the paper.
    In Section \ref{Preliminaries} we collect the necessary background material
    on curves on hyperbolic surfaces \ref{RefProject-Background_material},
    Dehn--Thurston coordinates \ref{RefProject_D-T_coord},
    the Bestvina--Bux--Margalit cycle complex \ref{Complex_B_B_M}, and the collar theorem \ref{subsec_collar}.
    The only new result in this section is Lemma \ref{le_vypuklost_B}, which establishes the convexity 
    of the length function on the cycle complex $\mathcal{B}_{1,2}$.
    In Section \ref{RefProject_General_case} we prove Theorems \ref{th_nezavisimost_description} and 
    \ref{th_general_description}.
    In Section \ref{Case S_1_2} we prove the main result of this work, namely Theorem 
    \ref{th_S_1_2_description}.
    
    \medskip
    \noindent\textbf{Acknowledgments.}
    I am deeply grateful to my advisor, Alexander Alexandrovich Gaifullin, for his invaluable guidance and
    support throughout this research. I am particularly thankful for his patience, insightful advice, and 
    constructive comments, which helped to structure the work and significantly improve its quality.
    I am also grateful to Maxim Alexandrovich Korolev for useful discussions.
    The author is a winner of the all-Russia mathematical August Moebius contest of graduate and 
    undergraduate student papers and thanks the jury and the board for the high praise of his work.

    \section{Preliminaries} \label{Preliminaries}
    We shall use the following standard notation. Let $S_{g,n}^{b}$ denote a two-dimensional 
    surface of genus $g$ with $b$ boundary components and $n$ punctures. For convenience,
    we sometimes use the abbreviated notation $S_g = S_{g,0}^{0}$, $S_{g, n} = S_{g,n}^{0}$, 
    and $S_{g}^{b} = S_{g,0}^{b}$.

   \subsection{Curves on surfaces}\label{RefProject-Background_material}
   Let $\gamma$ be an oriented simple closed curve on the surface $S_{g,n}^{b}$. We denote by
    $[\gamma]$ its homology class in $H_1(S_{g,n}^{b}, \mathbb{Z})$.

   For isotopy classes of simple closed curves $\alpha$ and $\beta$, we denote by 
   $\hat{i}(\alpha,\beta)$ the algebraic intersection number and by $i(\alpha,\beta)$ 
   the geometric intersection number.

    \begin{definition} \label{def_multicurve}
    A formal sum $\gamma = \sum^{k}_{i=1} a_i \gamma_i$ is called an \textit{integral multicurve} on $S_{g,n}^{b}$ 
    if the $\gamma_i$ are pairwise non-homotopic, simple, closed, essential, pairwise disjoint curves, 
    and $a_i \in \mathbb{Z}_{>0}$.
    A curve is called essential if it is not homotopic to a point, to a puncture, or to a boundary component 
    of $\partial S_{g,n}^{b}$.
    An integral multicurve is said to be \textit{oriented} if each of its components is equipped with an orientation.
    \end{definition}
    Let $\mathcal{MC}^{b}_{n, g}$ denote the set of isotopy classes of unoriented integral multicurves
    on the surface $S_{g,n}^{b}$. In what follows, we shall not distinguish between an integral multicurve
    and its isotopy class. We shall identify an integral multicurve $\gamma = \gamma_1$ with 
    the simple closed curve $\gamma_1$.
    
    \begin{definition}
    The length $\ell_{\gamma}$ of a homotopy class of a simple closed curve $\gamma$ on a hyperbolic surface
    $S_{g,n}$ is defined to be the length of its geodesic representative. This definition is well-posed, 
    since any homotopy class of a simple closed curve $\gamma$ contains a unique geodesic representative.
    For an integral multicurve $\gamma = \sum^{k}_{i=1} a_i \gamma_i$, we define its length by
    \[\ell_{\gamma} = \sum^{k}_{i=1} a_i\ell_{\gamma_i}. \]
    \end{definition}

    Let $\Homeo^{+}(S,\partial S)$ be the group of orientation-preserving homeomorphisms of the surface 
    that are the identity on the boundary.
    Let $\Homeo_0(S,\partial S)$ denote the identity component of this group. The mapping class group 
    $\Mod_{g,n} = \Mod(S_{g,n})$ is defined as
    \[
    \Mod_{g,n} = \Homeo^{+}(S,\partial S) / \Homeo_0(S,\partial S).
    \]

    \begin{theorem}[M. Mirzakhani \cite{Mirz}] \label{th_Mirz}
        Let $\gamma$ be an integral multicurve on the surface $S_{g,n}$ with a fixed hyperbolic metric $X$. 
        Let $\Mod_{g,n}(\gamma)$ denote the orbit of $\gamma$ under the action of the mapping class group. 
        Denote by $s_X(L, \gamma)$ the number of integral multicurves of length $\le L$
        belonging to $\Mod_{g,n}(\gamma)$. Then the following asymptotic holds:
        \[s_X(L, \gamma) \sim n_{\gamma}(X) L^{6(g-1)+2n},\]
        where $n_{\gamma}(X)$ is a positive function depending only on $X$ and $\gamma$.
    \end{theorem}
    Let $T_{\beta} \in \Mod_{g,n}$ denote the left Dehn twist along the curve $\beta$.

    \subsection{Dehn--Thurston coordinates} \label{RefProject_D-T_coord}

    We now recall the construction of Dehn--Thurston coordinates on the set $\mathcal{MC}_{g, n}^{b}$, following \cite{HP}.

\begin{construction}
Fix an orientation on $S_{g, n}^{b}$ and a pants decomposition $\mathcal{P} = \{\alpha_i\}_{i=1}^{3(g-1)+n+b}$,
where each $\alpha_i$ is a simple closed geodesic. On each $\alpha_i$, fix a closed arc $b_i$ and a point $v_i 
\in \alpha_i \setminus b_i$. Choose a small closed neighbourhood $A_i = U(\alpha_i)$ and identify it with 
the standard annulus $A = S^1 \times [-1,1]$ so that $\alpha_i = S^1 \times \{0\}$. 
Let $P_j$ denote the connected components of $S_{g, n}^{b} \setminus \bigcup_{i} \mathring{U}(\alpha_i)$. 
On each $P_j$, fix a family of arcs representing all possible topological types of intersection of
an integral multicurve with $P_j$. Denote this set of arcs by $h_j$.
We assume that the arcs in $h_j$ are pairwise non-homotopic, nontrivial, and that their endpoints lie in 
$\bigcup_i (b_i \times \{-1, 1\})$.

\begin{figure}[h]
    \centering
    \begin{subfigure}[b]{0.44\textwidth}
           \centering
           \begin{tikzpicture}=

              \begin{scope}[xscale=1.3, yscale=1.3]
                \begin{scope}[rotate=0]
                \draw [very thick] (0.866, 0.5)  to [ out = 120,  in = -80]  (0.7, 1.5)  to [ out = -160,  in = -20]  (-0.7, 1.5)  to [ out = -100,  in = 60]  (-0.866, 0.5) ;
                \draw [very thick]  (0.7, 1.5)  to [ out = 160,  in = 20]  (-0.7, 1.5);
                
                \draw [ultra thick , color = red]  (0.3, 1.39)  to [ out = -170,  in = -10]  (-0.3, 1.39);
                \draw [ thick , color = red] (0.3, 1.45)  to [out = -90,  in = 90] (0.3, 1.33);
                \draw [ thick , color = red] (-0.3, 1.45)  to [out = -90,  in = 90] (-0.3, 1.33);
                \draw [ thick]  (0.18, 1.35)  to [out = -90,  in = 150] (1.259, -0.519);

                \draw [ thick]  (0, 1.35)  to [out = -90,  in = 60] (-0.3, -1.05);
                \draw [ thick, dashed]  (-0.3, -1.05) to [out = 145,  in =-85] (-1.08, 0.28);
                \draw [ thick]  (-1.08, 0.255)  to [out = 20,  in = -95] (-0.25, 1.36);

                \end{scope}

                \begin{scope}[rotate=120]
                \draw [very thick] (0.866, 0.5)  to [ out = 120,  in = -80]  (0.7, 1.5)  to [ out = -160,  in = -20]  (-0.7, 1.5)  to [ out = -100,  in = 60]  (-0.866, 0.5) ;
                \draw [very thick]  (0.7, 1.5)  to [ out = 160,  in = 20]  (-0.7, 1.5);

                \draw [ultra thick , color = red]  (0.3, 1.39)  to [ out = -170,  in = -10]  (-0.3, 1.39);
                \draw [ thick , color = red] (0.3, 1.45)  to [out = -90,  in = 90] (0.3, 1.33);
                \draw [ thick , color = red] (-0.3, 1.45)  to [out = -90,  in = 90] (-0.3, 1.33);
                \draw [ thick]  (0.18, 1.35)  to [out = -90,  in = 150] (1.259, -0.519);

                \draw [ thick]  (0, 1.35)  to [out = -90,  in = 60] (-0.3, -1.05);
                \draw [ thick, dashed]  (-0.3, -1.05) to [out = 145,  in =-85] (-1.08, 0.28);
                \draw [ thick]  (-1.08, 0.255)  to [out = 20,  in = -95] (-0.25, 1.36);
                \end{scope}

                \begin{scope}[rotate=-120]
                \draw [very thick] (0.866, 0.5)  to [ out = 120,  in = -80]  (0.7, 1.5)  to [ out = -160,  in = -20]  (-0.7, 1.5)  to [ out = -100,  in = 60]  (-0.866, 0.5) ;
                \draw [very thick]  (0.7, 1.5)  to [ out = 160,  in = 20]  (-0.7, 1.5);

                \draw [ultra thick , color = red]  (0.3, 1.39)  to [ out = -170,  in = -10]  (-0.3, 1.39);
                \draw [ thick , color = red] (0.3, 1.45)  to [out = -90,  in = 90] (0.3, 1.33);
                \draw [ thick , color = red] (-0.3, 1.45)  to [out = -90,  in = 90] (-0.3, 1.33);
                \draw [ thick]  (0.18, 1.35)  to [out = -90,  in = 150] (1.259, -0.519);

                \draw [ thick]  (0, 1.35)  to [out = -90,  in = 60] (-0.3, -1.05);
                \draw [ thick, dashed]  (-0.3, -1.05) to [out = 145,  in =-85] (-1.08, 0.28);
                \draw [ thick]  (-1.08, 0.255)  to [out = 20,  in = -95] (-0.25, 1.36);
                \end{scope}
            \end{scope}
            \end{tikzpicture}
           \caption{The general case of the set $h_j$.}

   \label{fig_Coord_D_T_1}
    \end{subfigure}
    \begin{subfigure}[b]{0.44\textwidth}
           \centering
           \begin{tikzpicture}=
              \draw (1.1, -0.25) -- node[above] {$\alpha$} cycle ;
              \draw (-0.5, 1.4) -- node[left] {$P_{1}$} cycle ;
              \draw (-0.5, -1.4) -- node[left] {$P_{-1}$} cycle ;
              \draw (-2, -0) -- node[left] {$v$} cycle ;
              \draw (-0.3, -0.4) -- node[below] {$b$} cycle ;
              \draw (1, 1.15) -- node[above] {$h_{1}$} cycle ;
              \draw (1.15, -1.55) -- node[above] {$h_{-1}$} cycle ;

              \draw [very thick] (-2, 2)  to [ out = -90,  in = 90]  (-2, -2)
                                          to [out = -20,  in = 200]  (-1, -2)
                                          to [out = 15,  in = 165]  (1, -2)
                                          to [out = -20,  in = 200]  (2, -2)
                                          to [out = 90,  in = -90]  (2, 2);

              \draw [very thick] (-2, 2)  to [out = -20,   in = 200]  (-1, 2)
                                          to [out = -15,  in = 195]  (1,  2)
                                          to [out = -20,   in = 200]  (2,  2);
            \draw [very thick] (-2, 2)   to [out = 20,   in = 160]  (-1, 2);
            \draw [very thick] (1,  2)   to [out = 20,   in = 160]  (2,  2);

            \draw [very thick, dashed] (-2, -2)   to [out = 20,   in = 160]  (-1, -2);
            \draw [very thick, dashed] (1,  -2)   to [out = 20,   in = 160]  (2,  -2);

            \draw [very thick] (-2, 0)   to [out = -20,  in = 200] (2, 0);
            \draw [very thick, dashed] (-2, 0)   to [out = 20,   in = 160]  (2, 0);

            \draw [thick] (-2, 0.6)   to [out = -20,  in = 200] (2, 0.6);
            \draw [ thick, dashed] (-2, 0.6)   to [out = 20,   in = 160]  (2, 0.6);

            \draw [ thick] (-2, -0.6)   to [out = -20,  in = 200] (2, -0.6);
            \draw [ thick, dashed] (-2, -0.6)   to [out = 20,   in = 160]  (2, -0.6);

            \draw [very thick, dashed] (0.7, 1.92)   to [out = -15,  in = -165] (2, 1.8);
            \draw [very thick, dashed] (0.7, -1.92)   to [out = 15,  in = 165] (2, -1.8);

            \draw [very thick] (-0.5,  0.24)   to [out = 90,  in = -120] (0.7, 1.92);
            \draw [very thick] (-0.5, -0.97)   to [out = -90,  in = 120] (0.7, -1.92);

            \draw [very thick] (-0.25,  0.2)   to [out = 90,  in = -120] (2, 1.8);
            \draw [very thick] (-0.25, -1)   to [out = -90,  in = 120] (2, -1.8);
              
            \draw [ultra thick , color = red] (-0.7, -0.35)   to [out = -9,  in = 184] (0.3, -0.39);
            \draw [ thick , dashed,   color = red] (-0.7, -0.95)   to [out = 90,  in = -90] (-0.7, 0.3);
            \draw [ thick , dashed,   color = red] (0.3, -0.95)   to [out = 90,  in = -90] (0.3, 0.3);
            \draw[decorate, decoration={brace, amplitude=6pt}] 
                  (2.15,0.6) -- (2.15,-0.6) node[midway, right=8pt] {$A$};
            \filldraw[red] (-2, 0) circle (2 pt);  
           \end{tikzpicture}
           \caption{Construction of Dehn--Thurston coordinates in the case of $S_{0}^{4}$.}
   \label{fig_Coord_D_T}
    \end{subfigure}

\end{figure}

Let $\gamma$ be an integral multicurve on $S_{g, n}^{b}$.
If the parameter $m_i = 0$, then $\gamma$ contains $k \in \mathbb{N} \cup \{0\}$
parallel copies of the curve $\alpha_i$. In this case, the parameter $t_i$ is set equal to $k$.

We now consider the case where $m_i > 0$.
Assume that the integral multicurve $\gamma$ and the integral multicurve $\sum_i \alpha_i$ are in minimal position.
We isotop the integral multicurve $\gamma$ so that the arcs of $\gamma \cap P_j$ lie in small metric neighbourhoods of arcs from the set $h_j$ of the corresponding topological type, and such that $\gamma \cap \partial P_j \subset \bigcup_i (b_i \times \{-1, 1\})$.
Finally, we isotop $\gamma$ so that for any $\tau \in [-1, 1]$ and any curve $\alpha_i$, we have
$i(\gamma, \alpha_i \times \{\tau\}) = m_i$.

We now define the twist parameter $t_i$. The absolute value $|t_i|$ is equal to the minimal number of intersections of the arcs
of $\gamma \cap A_i$ with the segment $v_i \times [-1,1]$, where the minimum is taken over all isotopies of the integral multicurve $\gamma$ that are the identity outside $\alpha_i \times (-1,1)$. The sign of $t_i$ is positive if some arc of $\gamma \cap A_i$ twists to the right, and negative if some arc twists to the left. Note that arcs twisting in both directions cannot occur.

If the set $\{\alpha_{i_{1}}, \ldots, \alpha_{i_{q}} \}$ bounds a pair of pants in $S_{g, n}^{b}$, then the following equality holds:
\[
i(\alpha_{i_{1}}, \gamma) + \ldots + i(\alpha_{i_{q}}, \gamma) = m_{i_{1}} + \ldots + m_{i_{q}} \equiv \hat{i}(\alpha_{i_{1}}, \gamma) + \ldots + \hat{i}(\alpha_{i_{q}}, \gamma) \equiv 0 \pmod{2}.
\]
Note that, in general, $q$ can take any value in $\{1, 2, 3\}$.
\end{construction}

Let $Z(\mathcal{P}) \subset \left(\mathbb{Z}_{\ge 0} \times \mathbb{Z}\right)^{3(g-1)+n+b}$ be 
the set of tuples satisfying the following conditions:
\begin{enumerate}
  \item If $m_i = 0$, then $t_i \ge 0$.
  \item If the set $\{\alpha_{i_{1}}, \ldots, \alpha_{i_{q}} \}$ bounds a pair of pants in $S_{g, n}^{b}$, then the sum
  $m_{i_{1}} + \ldots + m_{i_{q}}$ is even.
\end{enumerate}
With this notation, the following theorem holds.
\begin{theorem}[Dehn, see \cite{HP}] \label{th_Den}
Consider the map  
 \begin{gather*}
   D : \mathcal{MC}^{b}_{n, g} \rightarrow Z(\mathcal{P}) \\
   \gamma \mapsto \bigl(m_1(\gamma), t_1(\gamma), \ldots, m_d(\gamma), t_d(\gamma)\bigr), 
 \end{gather*}
 where the coordinates $\bigl(m_1(\gamma), t_1(\gamma), \ldots, m_d(\gamma), t_d(\gamma)\bigr)$ are the parameters constructed above, and $d = 3(g-1)+n+b$.
 Then the map $D$ is a bijection.
\end{theorem}

\begin{remark}
 The Dehn--Thurston coordinates depend on the choice of the parameters
  \[\bigcup_i\bigl(\alpha_i, b_i, v_i, \{h_{j_{k}}\}_{k=1}^{s_i}\bigr).\]
\end{remark}

We shall be interested in a particular case. Let us list the features of the Dehn--Thurston coordinates for the set $\mathcal{MC}_{0}^{4}$. 
The coordinates are indexed by a pair 
$(m, t) \in 2\mathbb{N} \times \mathbb{Z} \cup \{0\} \times \mathbb{N}$.
Here $\mathbb{N}$ denotes the set of positive integers.
The pants decomposition $\mathcal{P} = \{\alpha\}$ consists of a single curve.
The set $S_{0}^{4} \setminus \mathring{U}(\alpha)$ consists of exactly two components $P_{-1}, P_1$.  
Each set $h_j$ consists of exactly one arc.
The set $Z(\mathcal{P}) = 2\mathbb{N} \times \mathbb{Z} \cup \{0\} \times \mathbb{N}$.
In this case, the Dehn--Thurston coordinates depend on the quintuple $(\alpha, b, v, h_{-1}, h_1)$.

Moreover, in what follows we shall use the fact that in any homotopy class of a simple closed curve 
on a hyperbolic surface with geodesic boundary, there exists a unique simple closed geodesic. 
This statement is a basic fact in the theory of hyperbolic surfaces; a proof can be found 
in \cite{Farb_Margalit}.  
Thus, the Dehn--Thurston coordinates parametrise \textit{integral multigeodesics}, and the 
length of an integral multicurve is identified with the length of the corresponding integral multigeodesic.
 
\begin{lemma} \label{le_ocenka_dliny_v_koordinatah_D_T}
   Let Dehn--Thurston coordinates corresponding to the quintuple $(\alpha, b, v, h_{-1}, h_1)$ 
   be fixed on the set $\mathcal{MC}_{0}^{4}$.
   Denote by $\gamma_{(m,t)}$ the integral multicurve on $S_{0}^{4}$ whose coordinates are $(m, t)$.
   Then the following inequality holds:
   \[\ell_{\gamma_{(m,t)}} \le \frac{m}{2}\left(\ell_{h_{-1}} +\ell_{h_1}\right) + |t|\ell_{\alpha}. \]
\end{lemma} 
\begin{proof}
   This lemma is a direct consequence of the construction of Dehn--Thurston coordinates
   described in Section \ref{RefProject_D-T_coord}, together with the fact that a simple closed geodesic
   is the shortest curve in its homotopy class.
\end{proof}

\begin{lemma} \label{le_o_harakterizacii_multicrivyh}
  Let Dehn--Thurston coordinates corresponding to the quintuple $(\alpha, b, v, h_{-1}, h_1)$ 
  be fixed on the set $\mathcal{MC}_{0}^{4}$.
  Denote by $\gamma_{(m,t)}$ the integral multicurve on $S_{0}^{4}$ whose coordinates are $(m, t)$.

  Then the integral multicurve $\gamma_{(m,t)}$ consists of
  $\Delta = \gcd\left(\frac{m}{2}, t\right)$ parallel
  copies of the simple closed geodesic
  $\gamma_{\left(\frac{m}{\Delta}, \frac{t}{\Delta}\right)}$.
\end{lemma}
\begin{proof}
  This fact follows directly from the construction of Dehn--Thurston coordinates.
\end{proof}

 In what follows, we shall be particularly interested in the two curves $\gamma_{(2,0)}$ and $\gamma_{(0,1)}$,
 which are illustrated in Figure \ref{fig_gamma_2_0_gamma_0_1}.
 
 \begin{figure}[h]
  \centering
  \begin{tikzpicture} 
    \draw (1.1, -0.25) -- node[below] {$\gamma_{(0, 1)}$} cycle ;
    \draw (0.4, 1.4) -- node[left] {$\gamma_{(2, 0)}$} cycle ;

    \draw [very thick] (-2, 2)  to [ out = -90,  in = 90]  (-2, -2)
                                to [out = -20,  in = 200]  (-1, -2)
                                to [out = 15,  in = 165]  (1, -2)
                                to [out = -20,  in = 200]  (2, -2)
                                to [out = 90,  in = -90]  (2, 2);

    \draw [very thick] (-2, 2)  to [out = -20,   in = 200]  (-1, 2)
                                to [out = -15,  in = 195]  (1,  2)
                                to [out = -20,   in = 200]  (2,  2);
   \draw [very thick] (-2, 2)   to [out = 20,   in = 160]  (-1, 2);
   \draw [very thick] (1,  2)   to [out = 20,   in = 160]  (2,  2);

   \draw [very thick, dashed] (-2, -2)   to [out = 20,   in = 160]  (-1, -2);
   \draw [very thick, dashed] (1,  -2)   to [out = 20,   in = 160]  (2,  -2);

   \draw [very thick] (-2, 0)   to [out = -20,  in = 200] (2, 0);
   \draw [very thick, dashed] (-2, 0)   to [out = 20,   in = 160]  (2, 0);

   \draw [thick] (-2, 0.6)   to [out = -20,  in = 200] (2, 0.6);
   \draw [ thick, dashed] (-2, 0.6)   to [out = 20,   in = 160]  (2, 0.6);

   \draw [ thick] (-2, -0.6)   to [out = -20,  in = 200] (2, -0.6);
   \draw [ thick, dashed] (-2, -0.6)   to [out = 20,   in = 160]  (2, -0.6);

   \draw [very thick] (-0.5, -0.39)   to [out = 90,  in = -120] (0.7, 1.92);
   \draw [very thick] (-0.5, -0.39)   to [out = -90,  in = 120] (0.7, -1.92);

   \draw [very thick, dashed] (0.7, 1.92)   to [out = -15,  in = -165] (2, 1.8);
   \draw [very thick, dashed] (0.7, -1.92)   to [out = 15,  in = 165] (2, -1.8);

   \draw [very thick] (-0.25, -0.39)   to [out = 90,  in = -120] (2, 1.8);
   \draw [very thick] (-0.25, -0.39)   to [out = -90,  in = 120] (2, -1.8);
    
   \draw [ultra thick , color = red] (-0.7, -0.35)   to [out = -9,  in = 184] (0.3, -0.39);
   \draw [ thick , dashed,   color = red] (-0.7, -0.95)   to [out = 90,  in = -90] (-0.7, 0.3);
   \draw [ thick , dashed,   color = red] (0.3, -0.95)   to [out = 90,  in = -90] (0.3, 0.3);
    
 \end{tikzpicture}

   \caption{The curves $\gamma_{(2,0)}$ and $\gamma_{(0,1)}$}
   \label{fig_gamma_2_0_gamma_0_1}
 \end{figure}
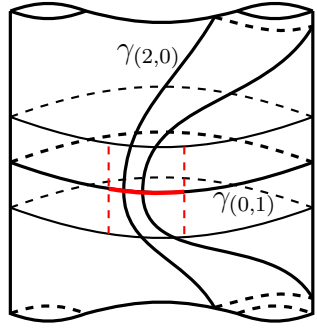
  We now describe, in terms of the coordinates $(m,t)$, the action of the left Dehn twist along the 
  geodesic $\gamma_{(0,1)}$ and the left half-twist along the geodesic $\gamma_{(2,0)}$ on integral
  multicurves.
 \begin{lemma} \label{le_deystvie_T_gamma_0_1}
    Let Dehn--Thurston coordinates corresponding to the quintuple $(\alpha, b, v, h_{-1}, h_1)$ be fixed 
    on the set $\mathcal{MC}_{0}^{4}$.
    Denote by $\gamma_{(m,t)}$ the integral multicurve on $S_{0}^{4}$ whose coordinates are $(m, t)$.
    Let $T_{\gamma_{(0,1)}}$ and $T_{\gamma_{(2,0)}}$ denote the left Dehn twists along the simple closed 
    curves $\gamma_{(0,1)}$ and $\gamma_{(2,0)}$.
    Then
 \begin{align}
    &T^{\pm 1}_{\gamma_{(0,1)}}\left(\gamma_{(m, t)}\right) = \gamma_{(m, t \mp m)}, \label{align_1} \\ 
    &T^{\pm1}_{\gamma_{(2,0)}}\left(\gamma_{(m, t)}\right) = \gamma_{\bigl(|m \pm 4t|, t \cdot \theta(m \pm 4t)\bigr)}, \label{align_2}
 \end{align}
 \begin{equation*}
          \text{ where $\theta(x) = $}  
       \begin{cases}
           1,& x > 0\\
          -1, &  x \le 0\\
       \end{cases}
 \end{equation*}
 \end{lemma}
\begin{proof}
Formula \eqref{align_1} is a direct consequence of the construction of Dehn--Thurston coordinates described
in Section \ref{RefProject_D-T_coord}, together with the fact that the simple closed geodesic $\alpha$ 
coincides with the integral multicurve $\gamma_{(0,1)}$ (the geodesic $\alpha$ was introduced in the 
construction of Dehn--Thurston coordinates; see Section \ref{RefProject_D-T_coord}).

We now turn to the proof of formula \eqref{align_2}.
The curve $\alpha' = \gamma_{(2, 0)}$ defines a pants decomposition of the surface $S_{0}^{4}$. 
We associate to this decomposition new Dehn--Thurston coordinates $(m', t')$ such that 
the coordinates of the curve $\alpha$ in this system are $(2, 0)$. We denote by $\gamma'_{(m', t')}$ 
the integral multicurve whose coordinates in the primed coordinate system are $(m', t')$.

We introduce the right surgery function $\mathrm{SR}: \mathcal{MC}_{0}^{4} \times \mathcal{MC}_{0}^{4} \rightarrow \mathcal{MC}_{0}^{4}$.
The action of this function on an ordered pair of integral multicurves $\left(a_1\gamma_1, a_2 \gamma_2\right)$ 
is as follows. Consider the union
of integral multicurves $a_1\gamma_1 \cup a_2 \gamma_2$,
placed in minimal position. Note that any element of $\mathcal{MC}_{0}^{4}$ 
can be represented as several parallel 
copies of a simple closed curve.

If $a_1 \ge a_2$, we resolve the intersections as shown in Figure \ref{fig_deystvie_psi}. 
Here the arcs of the integral multicurve $a_2\gamma_2$ turn to the right when approaching 
$a_1 \gamma_1$. For $a_1 < a_2$, we proceed analogously, but now $a_1 \gamma_1$ turns to 
the left when approaching $a_2\gamma_2$.  
The resulting integral multicurve is denoted by $\mathrm{SR}(a_1\gamma_1, a_2 \gamma_2)$. 
\begin{figure}[h] 
        \centering
    \begin{tikzpicture} 
       \draw (-2.8, 2) -- node[left] {$a_2 \gamma_2$} cycle ;
       \draw (-1.2, 1.5) -- node[right] {$a_1 \gamma_1$} cycle ;
       
       \draw (4.8, 1.5) -- node[right] {$\mathrm{SR}(a_1\gamma_1, a_2 \gamma_2)$} cycle ;

       \draw [ thick, ->] (0.3, 0) -- (1.5, 0);

       \draw [ thick, ] (-2.8, -2.5) -- (-2.8, 2.5);
       \draw [ thick, ] (-2.65, -2.5) -- (-2.65, 2.5);
       \draw [ thick, ] (-2.5, -2.5) -- (-2.5, 2.5);
       \draw [ thick, ] (-2.35, -2.5) -- (-2.35, 2.5);
       \draw [ thick, ] (-2.2, -2.5) -- (-2.2, 2.5);
       \draw [ thick, ] (-2.05, -2.5) -- (-2.05, 2.5);
       \draw [ thick, ] (-1.9, -2.5) -- (-1.9, 2.5);
     
       \draw [ thick, ] (-4.7, 0) -- (0, 0);
       \draw [ thick, ] (-4.7, 0.15) -- (0, 0.15);
       \draw [ thick, ] (-4.7, 0.3) -- (0, 0.3);
       \draw [ thick, ] (-4.7, 0.45) -- (0, 0.45);
       \draw [ thick, ] (-4.7, 0.6) -- (0, 0.6);
       \draw [ thick, ] (-4.7, 0.75) -- (0, 0.75);
       \draw [ thick, ] (-4.7, 0.9) -- (0, 0.9);
       \draw [ thick, ] (-4.7, 1.05) -- (0, 1.05);

       \draw [ thick, ] (-4.7, -0.15) -- (0, -0.15);
       \draw [ thick, ] (-4.7, -0.3) -- (0, -0.3);
       \draw [ thick, ] (-4.7, -0.45) -- (0, -0.45);
       \draw [ thick, ] (-4.7, -0.6) -- (0, -0.6);
       \draw [ thick, ] (-4.7, -0.75) -- (0, -0.75);
       \draw [ thick, ] (-4.7, -0.9) -- (0, -0.9);
       \draw [ thick, ] (-4.7, -1.05) -- (0, -1.05);

       \draw [ thick, ] (1.7, 0.25) --            (3.2, 0.25)-- (4.25, 1.25) -- (4.25, 2.5);
       \draw [ thick, ] (1.7, 0.4) --             (3.2, 0.4)-- (4.1, 1.25) -- (4.1, 2.5);
       \draw [ thick, ] (1.7, 0.55) --            (3.2, 0.55)-- (3.95, 1.25) -- (3.95, 2.5);
       \draw [ thick, ] (1.7, 0.7) --   (3.2, 0.7) -- (3.8, 1.25) -- (3.8, 2.5);
       \draw [ thick, ] (1.7, 0.85) --            (3.2, 0.85)-- (3.65, 1.25) -- (3.65, 2.5);
       \draw [ thick, ] (1.7, 1.0) --             (3.2, 1.0) -- (3.5,  1.25) -- (3.5, 2.5);
       \draw [ thick, ] (1.7, 1.15) --            (3.2, 1.15) -- (3.35, 1.25) -- (3.35, 2.5);

       \draw [ thick, ] (3.35, -2.5) -- (3.35, -1.25) --           (4.4, -0.25) -- (6, -0.25);
       \draw [ thick, ] (3.5, -2.5)  -- (3.5, -1.25)  --           (4.4, -0.4) -- (6, -0.4);
       \draw [ thick, ] (3.65, -2.5) -- (3.65, -1.25) --           (4.4, -0.55) -- (6, -0.55);
       \draw [ thick,] (3.8, -2.5) -- (3.8, -1.25) -- (4.4, -0.7) -- (6, -0.7);
       \draw [ thick, ] (3.95, -2.5) -- (3.95, -1.25) --           (4.4, -0.85) -- (6, -0.85);
       \draw [ thick, ] (4.1, -2.5) -- (4.1, -1.25) --            (4.4, -1.0) -- (6, -1.0);
       \draw [ thick, ] (4.25, -2.5) -- (4.25, -1.25) --           (4.4, -1.15) -- (6, -1.15);

       \draw [ thick, ] (1.7, 0) -- (3.2, 0) --     (4.4, 1.05) -- (6, 1.05);
       \draw [ thick, ] (1.7, -0.15) --       (3.2, -0.15) -- (4.4, 0.9) -- (6, 0.9);
       \draw [ thick, ] (1.7, -0.3) --        (3.2, -0.3) --  (4.4, 0.75) -- (6, 0.75);
       \draw [ thick, ] (1.7, -0.45) --       (3.2, -0.45) -- (4.4, 0.6) -- (6, 0.6);
       \draw [ thick, ] (1.7, -0.6) --        (3.2, -0.6) --  (4.4, 0.45) -- (6, 0.45);
       \draw [ thick, ] (1.7, -0.75) --       (3.2, -0.75) -- (4.4, 0.3) -- (6, 0.3);
       \draw [ thick, ] (1.7, -0.9) --        (3.2, -0.9 ) -- (4.4, 0.15) -- (6, 0.15);
       \draw [ thick, ] (1.7, -1.05) --       (3.2, -1.05) -- (4.4, 0) -- (6, 0) ;
       
    \end{tikzpicture}
     \caption{The action of $\mathrm{SR}(\cdot)$.}
 
     \label{fig_deystvie_psi}
 \end{figure}
It is easy to see that for $m > 0$, the following equalities hold:
\begin{equation} \label{eq_SR} 
    \gamma_{(m, t)} =   \begin{cases}
            \mathrm{SR}\left(t  \gamma_{(0, 1)}, \frac{m}{2} \gamma_{(2, 0)}\right),&  t > 0;\\
            \mathrm{SR}\left( \frac{m}{2} \gamma_{(2, 0)}, |t| \gamma_{(0, 1)}\right),&  t < 0.\\
       \end{cases}
\end{equation}
Since $\gamma'_{(0, 1)} = \gamma_{(2, 0)}$ and $\gamma'_{(2, 0)} = \gamma_{(0, 1)}$, the analogous statement in the primed coordinate system for $m' > 0$ is
\begin{equation*}  
       \gamma'_{(m', t')} =
        \begin{cases}
           \mathrm{SR}\left(t'  \gamma_{(2, 0)}, \frac{m'}{2} \gamma_{(0, 1)}\right), &  t' > 0;\\
           \mathrm{SR}\left(\frac{m'}{2} \gamma_{(0, 1)}, |t'| \gamma_{(2, 0)}\right), &  t' < 0.\\
       \end{cases}
\end{equation*}
Combining the last two systems, we obtain the following change of coordinates formula for $m t \neq 0$:
\begin{equation*}  
       \begin{cases}
           m' = 2|t|,\\
           t' = -\theta(t) \frac{m}{2}. \\
       \end{cases}
\end{equation*}
A direct check shows that the same change of coordinates formula holds when $m t = 0$.
According to formula \eqref{align_1}, in the new coordinates the twist $T_{\gamma_{(2, 0)}}$ acts as follows:
\[
      T_{\gamma_{(2, 0)}}^{\pm 1}\bigl(\gamma'_{(m', t')}\bigr) = \gamma'_{\left(m', t' \mp m'\right)}.
  \] 
To complete the proof of the lemma, it remains to apply the change of coordinates formula obtained above.
\end{proof}

\subsection{Cycle complex} \label{Complex_B_B_M}
We shall need an analogue of the cycle complex of Bestvina, Bux and Margalit for surfaces with boundary.
In the case of the torus with two boundary components, it was constructed in the original paper \cite{Bestvina}.
Since a detailed construction for arbitrary $S_{g,n}^{b}$ appears to be absent from the literature, we provide 
it in this section.

The term \textit{multicurve} used in the paper of Bestvina, Bux and Margalit \cite{Bestvina}
differs from the term integral multicurve in \cite{Mirz}.
Bestvina, Bux and Margalit call a \textit{multicurve} a nonempty finite collection of pairwise disjoint,
 pairwise non-isotopic, essential,
simple closed oriented curves (without attaching multiplicities to them). To avoid confusion,
we shall always refer to a multicurve in the sense of Bestvina, Bux and Margalit as an 
\textit{oriented multicurve}.
By essential curves we mean curves which are homotopically nontrivial and not homotopic to 
a loop around a puncture or a boundary component. In the definition from \cite{Mirz}, 
which we used in Section \ref{RefProject-Background_material}, the collection was replaced 
by a sum and finite linear combinations of curves with positive coefficients were allowed.
In this section, we follow the terminology of Bestvina, Bux and Margalit.

Let $S_{g,n}^{b}$ be an oriented surface with $g \ge 1$.
Denote by $\{\delta_i\}_{i=1}^{n+b}$ the set of pairwise distinct isotopy classes
of geodesics isotopic to boundary components or to loops around punctures.
We denote by $[\gamma] \in H_1(S_{1,2}, \mathbb{Z})$ the homology class of a curve $\gamma$.
Fix a primitive homology class $x \in H_1(S_{g,n}^{b}, \mathbb{Z}) / \langle [\delta_1], \ldots, [\delta_{n+b}] \rangle$.

Let $\mathcal{C}$ denote the set of isotopy classes of all oriented non-separating simple closed curves
on $S_{g,n}^{b}$. Consider the infinite-dimensional vector space $\mathbb{R}^{\mathcal{C}}$ consisting 
of finite formal linear combinations of elements of $\mathcal{C}$. Note that, in contrast to Definition \ref{def_multicurve}, 
we are now considering linear combinations of \textit{oriented} simple closed curves.

We define the \textit{cycle complex} $\mathcal{B}_{g,n}^{b} = \mathcal{B}_{g,n}^{b}(x)$ as the subset of $\mathbb{R}^{\mathcal{C}}$ consisting of all 
finite linear combinations $\gamma = \sum_{i=1}^{k} n_i \gamma_i$ (called \textit{cycles}) satisfying the following properties:
\begin{enumerate}
    \item $n_1, \ldots, n_k$ are positive integers;
    \item $\gamma_1, \ldots, \gamma_k$ are pairwise non-isotopic and pairwise disjoint oriented simple closed curves;
    \item $n_1 [\gamma_1] + \cdots + n_k [\gamma_k] = x$;
    \item The homology classes $[\gamma_1], \ldots, [\gamma_k]$ satisfy no \textit{one-sided relations}, that is, relations of the form $\sum_{i \in I} [\gamma_i]$, where $I$ is a nonempty subset of $\{1, \ldots, k\}$.
\end{enumerate}
A cycle $\gamma \in \mathcal{B}_{g,n}^{b}$ is called a \textit{basis cycle} for $x$ if the homology classes of the curves $\gamma_1, \ldots, \gamma_k$ are linearly independent 
in the group $x \in H_1(S_{g,n}^{b}, \mathbb{Z}) / \langle [\delta_1], \ldots, [\delta_{n+b}] \rangle$. Since the class $x$
is integral, the coefficients $n_i$ of any basis cycle for $x$ are positive integers.

We agree that for oriented multicurves, the notation $N \subset M$ means that the oriented multicurve $N$
is a formal sum of some components of the oriented multicurve $M$
with the same orientations that they have in $M$. 
For each cycle $\gamma \in \mathcal{B}_{g,n}^{b}$, the formal sum
of the oriented curves $\gamma_i$ that appear in $\gamma$ with nonzero coefficients
is an oriented multicurve. We call this oriented multicurve the
\textit{support} of the cycle $\gamma$.
Let $\mathcal{M}$ denote the set of all oriented multicurves on $S_{g,n}^{b}$
satisfying the following two conditions:
\begin{enumerate}
    \item For each component $\alpha$ of the oriented multicurve $M$, there exists a basis cycle
     for $x$ whose support is contained in $M$ and contains the curve $\gamma$.
    \item The homology classes of the components of the oriented multicurve $M$ satisfy no one-sided relation.
\end{enumerate}

Similarly to the closed surface case treated in \cite{Bestvina}, an oriented
multicurve $M$ belongs to $\mathcal{M}$ if and only if it
is the support of some cycle $\gamma \in \mathcal{B}_{g,n}^{b}$. 
Moreover, for each oriented
multicurve $M \in \mathcal{M}$, the set $P_M$ consisting of all cycles
$\gamma \in \mathcal{B}_{g,n}^{b}$ whose supports are contained in $M$ is a 
finite-dimensional convex polytope in $\mathbb{R}^{\mathcal{C}}$. 
The vertices of this polytope are precisely the basis cycles for $x$ whose supports are contained in $M$.
The convex polytopes $P_M$, where $M$ ranges over $\mathcal{M}$, form
a regular cell decomposition of the space $\mathcal{B}_{g,n}^{b}$.

We now define the dimension of a cell of the complex $\mathcal{B}_{g,n}^{b} = \mathcal{B}_{g,n}^{b}(x)$ 
and compute the dimension of the complex. Given $M \in \mathcal{M}$, let $|M|$ 
denote the number of curves in the oriented multicurve $M$,
let $D = D(M)$ be the dimension of the linear span of the images of the simple curves from $M$ in 
$H_1(S_{g,n}^{b}, \mathbb{Z}) / \langle [\delta_1], \ldots, [\delta_{n+b}] \rangle$, 
let $N$ be the number of components of $S \setminus M$, and let $B = B(M)$ 
be the dimension of the corresponding cell in $\mathcal{B}_{g,n}^{b}$. We shall need the following lemma, 
which generalizes Lemma $2.1$ of \cite{Bestvina} to the case of a not necessarily closed surface.
Its proof carries over verbatim to this setting.
\begin{lemma} \label{le_o_razmernosti_B_1_2}
    For any oriented multicurve $M \in \mathcal{M}$, we have $B = |M| - D = N - 1$,
    and the dimension of $\mathcal{B}_{g,n}^{b}$ is $2g - 2 + n + b$.
\end{lemma}

\begin{figure}[h]
        \centering
        \begin{tikzpicture} 
            \draw (-1.1, -0.2) -- node[below] {$\gamma_1$} cycle ; 
            \draw (1.1, -0.2) -- node[below] {$\gamma_2$} cycle ;
     
            \draw [ultra thick] (2, 1.5) to [ out = 180, in = 90]  (-2, 0)
                                        to [out = -90, in = 180]  (2, -1.5)
                                        to [out = 135,  in = -135]  (2, -0.5)
                                        to [out = 115,  in = -115]   (2, 0.5)
                                        to [out = 135,  in = -135]  (2, 1.5);
            \draw [ultra thick] (2.04, -1.51)  to [out = 45,  in = -45]  (2, -0.5)
                                          to [out = 115,  in = -115]   (2, 0.5)
                                          to [out = 45,  in = -45]  (2, 1.5);
            \draw [very thick] (-0.5, 0)  to [out = 90,  in = 90]   (0.5, 0)
                                          to [out = -90,  in = -90]  (-0.5, 0); 
            \draw [thick, -{stealth}] (-2, 0)  to [out = -30,  in = 180]   (-1.24, -0.22);
            \draw [thick, {stealth}-] (1.19, -0.2)  to [out = 0,  in = -150]   (1.89, 0);
            \draw [thick] (-2, 0)  to [out = -30,  in = -150]   (-0.5, 0);
            \draw [thick] (0.5, 0)  to [out = -30,  in = -150]   (1.89, 0);
    
            \draw [ thick, dashed] (-2, 0)  to [out = 30,  in = 150]   (-0.5, 0);
            \draw [ thick, dashed] (0.5, 0)  to [out = 30,  in = 150]   (1.89, 0);
            
         \end{tikzpicture}

         \caption{A top-dimensional cell in the case of $S_{1}^{2}$.}
     \label{fig_le_o_razmernosti_B_1_2}
\end{figure}
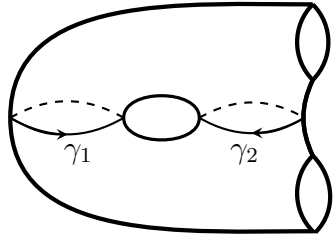

For the proofs and statements in this section, we shall need two functions: $\mathrm{Drain}$ and $\mathrm{Surger}$.
Given two cycles placed in minimal position, the function $\mathrm{Surger}$
acts on their union as shown in Figure \ref{fig_deystvie_Surger}. 
In contrast to Figure \ref{fig_deystvie_psi}, the resolution of the intersection is 
dictated by \textit{respecting the orientations} of the cycles. The construction of 
$\mathrm{Surger}$ in this setting is exactly the same as the case described in \cite[Section 5]{Bestvina}.

\begin{figure}[h]
        \centering
    \begin{tikzpicture} 
       \draw (-2.8, 2) -- node[left] {$k_j$} cycle ;
       \draw (-1.2, 1.5) -- node[right] {$k_i$} cycle ;
       \draw (3.2, 2) -- node[left] {$k_j$} cycle ;
       \draw (4.8, 1.5) -- node[right] {$k_i-k_j$} cycle ;
       \draw (5, -2.2) -- node[left] {$k_j$} cycle ;
       
       \draw [ thick, ->] (0.3, 0) -- (1.5, 0);

       \draw [ thick, ] (-2.8, -2.5) -- (-2.8, 2.5);
       \draw [ thick, ] (-2.65, -2.5) -- (-2.65, 2.5);
       \draw [ thick, ] (-2.5, -2.5) -- (-2.5, 2.5);
       \draw [ thick,-{stealth} ] (-2.35, -2.5) -- (-2.35, 2.5);
       \draw [ thick, ] (-2.2, -2.5) -- (-2.2, 2.5);
       \draw [ thick, ] (-2.05, -2.5) -- (-2.05, 2.5);
       \draw [ thick, ] (-1.9, -2.5) -- (-1.9, 2.5);
     
       \draw [ thick, -{stealth}] (-4.7, 0) -- (0, 0);
       \draw [ thick, ] (-4.7, 0.15) -- (0, 0.15);
       \draw [ thick, ] (-4.7, 0.3) -- (0, 0.3);
       \draw [ thick, ] (-4.7, 0.45) -- (0, 0.45);
       \draw [ thick, ] (-4.7, 0.6) -- (0, 0.6);
       \draw [ thick, ] (-4.7, 0.75) -- (0, 0.75);
       \draw [ thick, ] (-4.7, 0.9) -- (0, 0.9);
       \draw [ thick, ] (-4.7, 1.05) -- (0, 1.05);

       \draw [ thick, ] (-4.7, -0.15) -- (0, -0.15);
       \draw [ thick, ] (-4.7, -0.3) -- (0, -0.3);
       \draw [ thick, ] (-4.7, -0.45) -- (0, -0.45);
       \draw [ thick, ] (-4.7, -0.6) -- (0, -0.6);
       \draw [ thick, ] (-4.7, -0.75) -- (0, -0.75);
       \draw [ thick, ] (-4.7, -0.9) -- (0, -0.9);
       \draw [ thick, ] (-4.7, -1.05) -- (0, -1.05);

       \draw [ thick, ] (1.7, 0.25) --            (3.2, 0.25)-- (4.25, 1.25) -- (4.25, 2.5);
       \draw [ thick, ] (1.7, 0.4) --             (3.2, 0.4)-- (4.1, 1.25) -- (4.1, 2.5);
       \draw [ thick, ] (1.7, 0.55) --            (3.2, 0.55)-- (3.95, 1.25) -- (3.95, 2.5);
       \draw [ thick,-{stealth} ] (1.7, 0.7) --   (3.2, 0.7) -- (3.8, 1.25) -- (3.8, 2.5);
       \draw [ thick, ] (1.7, 0.85) --            (3.2, 0.85)-- (3.65, 1.25) -- (3.65, 2.5);
       \draw [ thick, ] (1.7, 1.0) --             (3.2, 1.0) -- (3.5,  1.25) -- (3.5, 2.5);
       \draw [ thick, ] (1.7, 1.15) --            (3.2, 1.15) -- (3.35, 1.25) -- (3.35, 2.5);

       \draw [ thick, ] (3.35, -2.5) -- (3.35, -1.25) --           (4.4, -0.25) -- (6, -0.25);
       \draw [ thick, ] (3.5, -2.5)  -- (3.5, -1.25)  --           (4.4, -0.4) -- (6, -0.4);
       \draw [ thick, ] (3.65, -2.5) -- (3.65, -1.25) --           (4.4, -0.55) -- (6, -0.55);
       \draw [ thick,-{stealth} ,] (3.8, -2.5) -- (3.8, -1.25) -- (4.4, -0.7) -- (6, -0.7);
       \draw [ thick, ] (3.95, -2.5) -- (3.95, -1.25) --           (4.4, -0.85) -- (6, -0.85);
       \draw [ thick, ] (4.1, -2.5) -- (4.1, -1.25) --            (4.4, -1.0) -- (6, -1.0);
       \draw [ thick, ] (4.25, -2.5) -- (4.25, -1.25) --           (4.4, -1.15) -- (6, -1.15);

       \draw [ thick, ] (1.7, 0) -- (3.2, 0) --     (4.4, 1.05) -- (6, 1.05);
       \draw [ thick, ] (1.7, -0.15) --       (3.2, -0.15) -- (4.4, 0.9) -- (6, 0.9);
       \draw [ thick, ] (1.7, -0.3) --        (3.2, -0.3) --  (4.4, 0.75) -- (6, 0.75);
       \draw [ thick, -{stealth}] (1.7, -0.45) --       (3.2, -0.45) -- (4.4, 0.6) -- (6, 0.6);
       \draw [ thick, ] (1.7, -0.6) --        (3.2, -0.6) --  (4.4, 0.45) -- (6, 0.45);
       \draw [ thick, ] (1.7, -0.75) --       (3.2, -0.75) -- (4.4, 0.3) -- (6, 0.3);
       \draw [ thick, ] (1.7, -0.9) --        (3.2, -0.9 ) -- (4.4, 0.15) -- (6, 0.15);
       \draw [ thick, ] (1.7, -1.05) --       (3.2, -1.05) -- (4.4, 0) -- (6, 0) ;
       
    \end{tikzpicture}
    \caption{The action of $\mathrm{Surger}(\cdot)$.}
 
     \label{fig_deystvie_Surger}
 \end{figure}

The function $\mathrm{Drain}$ is defined analogously to the case described in 
\cite[Section 5]{Bestvina}. 
Let $\gamma = \sum_{i} a_i \gamma_i$ be a cycle. Denote by $\{R_i\}$ the set of oriented 
embedded subsurfaces of $S_{g,n}^{b}$
(the orientation on $R_i$ is the same as that on $S_{g,n}^{b}$)
whose oriented boundaries belong to the support of the cycle $\gamma$. 
Let $\{R_1, \ldots, R_k\}$ be the subset of $\{R_i\}$ consisting of subsurfaces 
whose boundaries give rise to a one-sided relation of the form
\[\partial R_i = [\gamma_{i_{1}}] + \cdots + [\gamma_{i_{p}}] = 0, \]
where the equality is taken in the group $H_1(S_{g,n}^{b}, \mathbb{Z}) / \langle [\delta_1], \ldots, [\delta_{n+b}] \rangle$.
Then, for sufficiently small $\varepsilon$, we may consider the new positive cycle
\[
    \gamma - \varepsilon \left(\partial R_1 + \partial R_2 + \cdots + \partial R_k \right).
\]
We increase $\varepsilon$ until one of the coefficients of the cycle $\gamma$ becomes zero.
We then repeat the process with the new cycle. When no one-sided relations remain, we stop.
The resulting cycle is denoted by $\mathrm{Drain}(\gamma)$.
Since no arbitrary choices were made, the cycle $\mathrm{Drain}(\gamma)$ is uniquely defined.

\begin{theorem} \label{th_o_topologii_kompleksa_ciklov_general}
    The cycle complex $\mathcal{B}_{g,n}^{b}$ is contractible. 
\end{theorem}
\begin{proof}
The proof repeats verbatim the argument given in \cite[Section 5]{Bestvina}.      
The function $H(w, t) = \mathrm{Drain}\bigl(\mathrm{Surger}(tc + (1 - t)d)\bigr)$, constructed analogously, is again a strong deformation retraction of $\mathcal{B}_{g,n}^{b}$ onto the point $c$.
\end{proof}

\begin{lemma} \label{le_neravenstvo_H}
    Let a hyperbolic metric $X$ on the surface $S_{g,n}^{b}$ be fixed. 
    Consider the map
   \[
    H(v, w, t) : \mathcal{B}_{g,n}^{b} \times \mathcal{B}_{g,n}^{b} \times [0,1] \to \mathcal{B}_{g,n}^{b}, 
   \]
   \[
    H(v, w, t) = \mathrm{Drain}(\mathrm{Surger}(tc + (1 - t)d)), 
   \]
   where $c$ and $d$ are $1$-cycles representing the points $v \in \mathcal{B}_{g,n}^{b}$ and
   $w \in \mathcal{B}_{g,n}^{b}$, respectively.
   Then the following inequality holds:
   \[\ell_{\mathrm{Drain}(\mathrm{Surger}(tc + (1 - t)d))} \le t\ell_{c} + (1-t)\ell_{d},\]
   where $\ell_{\sum k_i \gamma_i} = \sum k_i \ell_{\gamma_i}$.  
\end{lemma}

We shall be particularly interested in the cycle complex $\mathcal{B}_{1, 2}$. In this case, an oriented multicurve 
can have only one component. Therefore, the vertices of the complex $\mathcal{B}_{1, 2}$ 
are the oriented 
simple closed curves belonging to the class 
\[x \in H_1(S_{g,n}^{b}, \mathbb{Z}) / \langle [\delta_1], \ldots, [\delta_{n+b}] \rangle. \]
From Lemma \ref{le_o_razmernosti_B_1_2} it follows that $\dim \mathcal{B}_{1, 2} = 1$. Therefore, from Theorem 
\ref{th_o_topologii_kompleksa_ciklov_general} we obtain the following.
\begin{corollary} \label{th_o_topologii_kompleksa_ciklov}
    The cycle complex $\mathcal{B}_{1, 2}$ is a tree.
\end{corollary}

\begin{lemma}\label{le_vypuklost_B}
   Let a hyperbolic metric $X$ on the surface $S_{1, 2}$ 
   with geodesic boundary components be fixed. Consider an arbitrary vertex
   $\gamma_1 \in \mathcal{B}_{1, 2}$. Let $\gamma_0, \gamma_2$ be vertices adjacent 
   to $\gamma_1$. Then 
   \[\ell_{\gamma_1} \le \frac{\ell_{\gamma_0} + \ell_{\gamma_2}}{2}.\]
\end{lemma}
\begin{proof}
    From the definition of $\mathcal{B}_{1, 2}$, it follows that the cycles corresponding to adjacent vertices are 
disjoint; hence, $\gamma_0$ and $\gamma_2$ lie in the subsurface $S_{0,2}^{2} \subset S_{1, 2}$ obtained 
by cutting along $\gamma_1$. Order the intersection points of the geodesics $\gamma_0$ and $\gamma_2$ using
the orientation on $\gamma_0$. Note that the signs of the intersection points alternate. This follows from the fact that the geodesic $\gamma_2$
is a separating curve on the surface $S_{0,2}^{2}$. Consider the intersection of the geodesics $\gamma_0$ and $\gamma_2$,
as shown in Figure \ref{fig_vypuklost}.

\begin{figure}[h]
    \centering
    \begin{tikzpicture} 
       \draw (4, 0) -- node[below] {$\gamma_0$} cycle ; 
       \draw (-3, -2) -- node[left] {$\gamma_2$} cycle ;
       \draw (3.5, 1.7) -- node[right] {$\mathrm{Surger}(\gamma_0+\gamma_2)$} cycle ;

       \draw [very thick, -{stealth}] (-5, 0) -- (5, 0);
       \draw [very thick, -{stealth}] (-3, -3) -- (-3, 3);
       \draw [very thick, -{stealth}] (-1, 3) -- (-1, -3);
       \draw [very thick, -{stealth}] (1, -3) -- (1, 3);
       \draw [very thick, -{stealth}] (3, 3) -- (3, -3);

       \draw [very thick, -{stealth}] (-2.8, -3) to [in = 180 , out = 90] (-2, -0.2);
       \draw [very thick] (-2.1, -0.2) to [out = 0, in = 90 ] (-1.2, -3);

       \draw [very thick, -{stealth}] (1.2, -3) to [in = 180 , out = 90] (2, -0.2);
       \draw [very thick] (1.9, -0.2) to [out = 0, in = 90 ] (2.8, -3);
       
       \draw [very thick, -{stealth}] (-0.8, 3) to [ out = -90, in = 180] (0, 0.2);
       \draw [very thick] (-0.1, 0.2) to [out = 0, in = -90 ] (0.8, 3);

       \draw [very thick, -{stealth}] (-5, 0.2) to [ out = 0, in = -90] (-3.2, 3);

       \draw [very thick, -{stealth}] (3.2, 3) to [ out = -90, in = 180] (5, 0.2);
    \end{tikzpicture}

     \caption{Intersection of the geodesics $\gamma_0$ and $\gamma_2$.}
     \label{fig_vypuklost}
\end{figure}
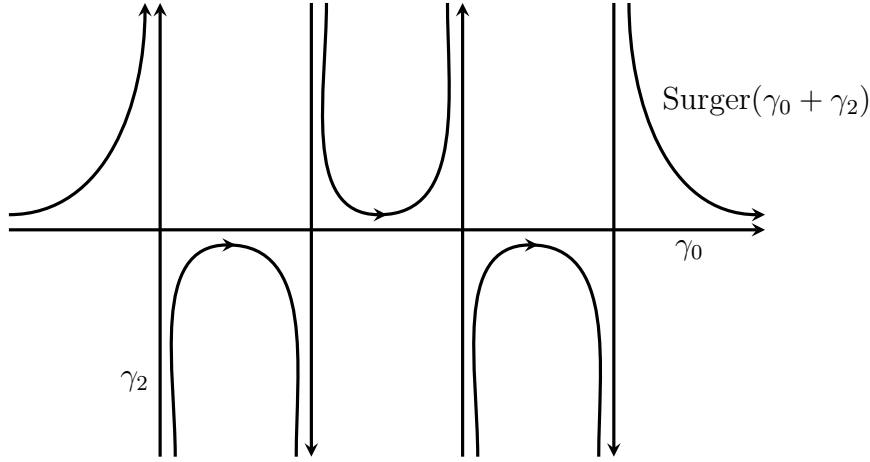 

Consider the cycle $\omega = \mathrm{Drain}(\mathrm{Surger}(\gamma_0 + \gamma_2))$. From Figure \ref{fig_vypuklost} 
it follows that the support of $\omega$ has zero geometric intersection number with $\gamma_0$.
Moreover, it lies on the surface $S_{0,2}^{2}$ obtained by cutting along $\gamma_1$.
Since $\mathrm{Drain}$ discards
loops around punctures on the surface $S_{1, 2}$, we obtain that the support of the cycle $\omega$ coincides with the oriented 
curve $\gamma_1$.  
Furthermore, $[\omega] = 2x$, where $x \in H_1(S_{1, 2}; \mathbb{Z}) / \langle [\delta_1] \rangle$.
This follows from the definition of $\mathcal{B}_{1, 2}$ and the fact that $\mathrm{Drain}(\mathrm{Surger}(\cdot))$ preserves the homology class of the cycle
in the group $H_1(S_{1,2}, \mathbb{Z}) / \langle [\delta_1], [\delta_{2}] \rangle$.
Hence, $\omega = 2 \gamma_1.$ 
Using Lemma \ref{le_neravenstvo_H}, we obtain the inequality 
$\ell_{\gamma_1} \le \frac{\ell_{\gamma_0} + \ell_{\gamma_2}}{2}$.
\end{proof}

\begin{corollary} \label{le_o_topologii_kompleksa_ciklov}
Let a hyperbolic metric $X$ on the surface $S_{1, 2}$ 
with geodesic boundary components be fixed. Consider a simple path in $\mathcal{B}_{1, 2}$ 
with vertices represented by geodesics $\gamma_1, \ldots, \gamma_n$, 
where $\gamma_1$ is a shortest representative of the homology class 
$x$. Then from the inequality
$1 \le s \le t \le n$ it follows that $\ell_{\gamma_s} \le \ell_{\gamma_t}$.
\end{corollary}

\subsection{The collar theorem and its consequences}\label{subsec_collar}
Let $\mathbb{H}^{2}$ denote the Lobachevsky plane (the hyperbolic plane).
Let $\alpha, \beta \subset \mathbb{H}^{2}$ be geodesics intersecting at an ideal point $c$.
Denote by $p_{\beta}$ the point of intersection of the geodesic $\beta$ with the boundary at infinity.
Drop a perpendicular $s$ from $p_{\beta}$ to $\alpha$ and denote its foot by $q$.
Then there exists a unique horocycle $h$ with 
center $c$ passing through $q$. 
Let $M$ be the closed region consisting of all points inside the horocycle $h$, 
and let $G$ be the closed region consisting of all points lying between the geodesics $\alpha$ and $\beta$.   
In the upper half-plane model, as in Figure \ref{fig_spike}, $M$ is the half-plane
given by the inequality $\operatorname{Im} z \ge \operatorname{Im} q$, and $G$ is the strip
given by the inequalities $\operatorname{Re} \alpha \le \operatorname{Re} z \le \operatorname{Re} \beta$.
Take two copies of the set $\mathscr{A} = M \cap G$ and identify the geodesic segments in them. 
The resulting cylindrical surface will be denoted by $\mathscr{S}$.
We shall call any region isometric to $\mathscr{S}$ a \textit{spike}, 
and any region isometric to $\mathscr{A}$ a \textit{half-spike}.
The construction is illustrated in Figure \ref{fig_spike}.
\begin{figure}[h]
    \centering
    \begin{tikzpicture}
        \fill (-0.5, 6.5) node {$\mathscr{A} $};
        \fill (-3.7, 6.6) node {$M$};
        \fill (-0.5, 2.7) node {$G$};
        \fill (0.95, 4) node {$s$};
        \fill (2.8, 5.7) node {$h$};
        \fill (-3.36, 5.2) node {$q$};
        \fill (-2.96, 3.2) node {$\alpha$};
        \fill (2.44, 3.5) node {$\beta$};
        \fill (-3.16, 5.39) circle (2pt);
        \fill (2.24, 0) circle (2pt);
        \fill (2.24, -0.3) node {$p_{\beta}$};

        \draw[very thick, dashed] (0: 2.24 cm) arc (0: 90: 5.39 cm);

        \draw[very thick, dotted] (-4, 5.39) --  (3.07, 5.39);
        \draw[very thick] (-4, 0) --  (3.07, 0);

        \draw[very thick] (-3.16, 0) --  (-3.16, 7);
        \draw[very thick] (2.24, 0) --  (2.24, 7);
         
        \draw (-3.16, 5.19) --  (-2.96, 5.19) -- (-2.96, 5.39);
        
    \end{tikzpicture}
    \caption{Construction of a spike.}
    \label{fig_spike}
\end{figure}
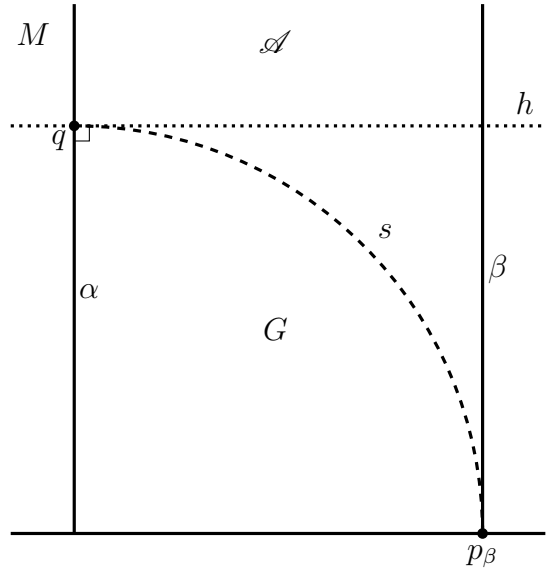

\begin{theorem}[\protect{Collar theorem in the noncompact case \cite[Section 4.4]{Bus}}] \label{th_collar}
    Let $S = S_{g, n}$ be a hyperbolic surface. Then 
    \begin{enumerate}
        \item $S$ contains uniquely determined pairwise disjoint \textit{spikes} $\mathscr{S}^{1}, \ldots, \mathscr{S}^{n}$,
              which are neighborhoods of the punctures.
        \item Let $\alpha_{1}, \ldots, \alpha_{3g-3+n}$ be a collection of geodesics on $S$ defining a pants decomposition.
        Then the collars $\mathscr{C}(\alpha_i) = \{p \in S \mid \sinh(\operatorname{dist}(p, \alpha_i)) \sinh (\tfrac{1}{2} \ell_{\alpha_i}) \le 1\}$ 
        around the geodesics $\alpha_{1}, \ldots, \alpha_{3g-3+n}$ are pairwise disjoint and do not intersect the spikes $\mathscr{S}^{1}, \ldots, \mathscr{S}^{n}$.         
    \end{enumerate} 
\end{theorem}
\begin{remark}
    This theorem extends readily to the case of a surface with geodesic boundary components $S_{g, n}^{b}$. 
    To see this, it suffices to take the double.
\end{remark}
\begin{corollary}\label{cor_collar}
    Let $S = S_{g, n}^{b}$ be a hyperbolic surface with geodesic 
    boundary components $\beta_{1}, \ldots, \beta_{b}$. 
    If $\gamma$ is a simple closed geodesic on $S$, then 
    \[
    \gamma \subset  \Omega(S), 
    \] 
\end{corollary}
where $\Omega(S)$ denotes the closure of the set  
\[
   S \setminus \left( \mathscr{S}^{1} \cup \ldots \cup \mathscr{S}^{n} \cup \mathscr{C}(\beta_1)\ldots \cup \mathscr{C}(\beta_{k})\right) .
\]

\begin{construction} \label{con_Psi}
    Let $S_{\infty}$ be a copy of $S \setminus \partial S$. 
    We regard $S_{\infty}$ as a punctured surface and equip it with a
    complete finite-area hyperbolic metric $X_{\infty}$. 
    Next, choose a diffeomorphism
    \[
    \Psi: \Omega(S) \to \Omega(S_{\infty}),
    \] 
    homotopic to the identity map $\Omega(S) \to \Omega(S)$ in the class 
    of maps from $\Omega(S)$ to $S$.
\end{construction}

\begin{remark}
   From Corollary \ref{cor_collar} it follows that the support of any simple closed geodesic 
on the hyperbolic surface $S$ is a subset of $\Omega(S)$. Therefore, for any simple
closed curve, the restriction of $\Psi$ to the support of this curve is a well-defined map.
The analogous statement holds for $\Psi^{-1}$.
\end{remark}
\begin{lemma}\label{le_svoystva_Psi}
    The map $\Psi$ has the following properties:
    \begin{enumerate}
        \item $\Psi$ preserves the homology class of a curve;
        \item $\Psi$ preserves the topological type of a curve.
    \end{enumerate}
\end{lemma}

\begin{lemma} \label{le_ocenka_general_Mirz}
    For any simple closed geodesic $\gamma$ on the surface $S_{g, n}^{b}$, 
    the following inequalities hold:
    \begin{align}
        \frac{1}{C} \ell_{\gamma}(X) \le \ell_{\Psi(\gamma)}(X_{\infty}) \le C \ell_{\gamma}(X),
    \end{align}
    where $C = C(X) > 0$ is a positive constant depending only on the metric $X$.
\end{lemma}
\begin{proof} 
   By the compactness of $\Omega(S)$, the norms $\|d\Psi\|$ and $\|d\Psi^{-1}\|$ 
are bounded above by some positive constant $C$. Hence, under these maps, 
the length of any simple closed curve is distorted by a factor of at most $C$.
\end{proof}

\section{General results}
\label{RefProject_General_case}
\subsection{The growth rate of lengths of simple closed geodesics for surfaces with boundary}
We formulate and prove a corollary of Mirzakhani's results \cite{Mirz} for the case of surfaces with boundary.
Let $S_{g, n}^{b}$ denote a surface of genus $g$ with $n$ punctures and $b$ geodesic boundary components.
Consider a complete finite-area hyperbolic metric $X$ on $S_{g, n}^{b}$ such that
the boundary components of $S_{g, n}^{b}$ are geodesic.

Consider the following direct generalisation of Mirzakhani's counting function to the case of surfaces with boundary:
\[s_{g, n}^{b}(L, \gamma) = \# \bigl\{\alpha \in \mathrm{Mod}_{g, n}^{b} \cdot \gamma \mid \ell_{\alpha}(X) \le L\bigr\}.\]
With this notation, the following theorem holds.

\begin{theorem} \label{th_general_Mirz}
   For any hyperbolic surface $S_{g, n}^{b}$ with metric $X$ and any geodesic $\gamma$
   not homotopic to a boundary component, there exists a constant $c = c(X) > 0$
   such that for all sufficiently large $L$ the following inequality holds:
   \[\frac{1}{c} L^{6g-6+2(n+b)} \le s_{g, n}^{b}(L, \gamma) \le c L^{6g-6+2(n+b)}. \]
\end{theorem}
\begin{proof}
    From Lemmas \ref{le_ocenka_general_Mirz} and \ref{le_svoystva_Psi} we obtain the inequalities
    \[s_{g, n+b}\left(\frac{L}{C}, \gamma\right) \le s_{g, n}^{b}(L, \gamma) \le s_{g, n+b}(C L, \gamma).\]
    To complete the proof, it remains to apply Theorem \ref{th_Mirz}:
    \[s_{g, n+b}(L, \gamma) \sim c(S_{g, n+b}) L^{6g-6+2(n+b)}.\]
\end{proof} 
\begin{remark}
    The result stated in Theorem \ref{th_general_Mirz} is weaker than Mirzakhani's original result, 
    but it includes the case of surfaces with boundary components. We are not aware of any work in which 
    an extension of Mirzakhani's theorem with a precise asymptotic is proved for hyperbolic 
    surfaces with boundary.
\end{remark}

\subsection{Proof of Theorem \ref{th_nezavisimost_description}}
In this section, we will prove Theorem \ref{th_nezavisimost_description}. It will immediately follow 
from Propositions \ref{th_nezavisimost_ot_metriki}, \ref{th_ne_zavisimosti_assimptotik_ot_klassa_gomologiy}, 
and \ref{th_nezavisimost_ot_b}.  
Let $x \in H_{1}\left(S_{g, n}^{b}, \mathbb{Z}\right)$ be a nonzero strictly primitive homological class.
We repeat the definition of the $h$ function given in the introduction,
\[
	h_{g,n}^{b}\left(L, x, X\right) = \# \left\{ \alpha \in x \mid  \ell_{\alpha}(X) \le L\right\}, 
\] 
where the symbol $\alpha \in x$ means that the curve $\alpha$ is an oriented simple
closed geodesic belonging to the class $x$.

\begin{proposition} \label{th_nezavisimost_ot_metriki}
   For any two hyperbolic metrics $X_1, X_2$ on the surface $S_{g, n}^{b}$ and any
   primitive homology class $x \in H_{1}(S_{g, n}^{b}, \mathbb{Z})$,
   there exists a constant $c = c(X_1, X_2) > 0$
   such that for all sufficiently large $L$ the following inequality holds:
   \[ h_{g,n}^{b}\left(\frac{L}{c}, x, X_1\right) \le h_{g,n}^{b}(L, x, X_2) \le h_{g,n}^{b}(c L, x, X_1). \]
\end{proposition}
\begin{proof}
	Denote by $\ell_{\gamma}(X_i)$ the length of the geodesic $\gamma$ in the metric $X_i$.
    From Wolpert's lemma (see \cite{Farb_Margalit}) we have the inequality
	\[\frac{1}{c} \, \ell_{\gamma}(X_1) \le \ell_{\gamma}(X_2) \le c \, \ell_{\gamma}(X_1),
    \quad \text{where } c = c(X_1, X_2) > 0.\]
	The required estimates follow immediately.
\end{proof}

\begin{proposition}\label{th_ne_zavisimosti_assimptotik_ot_klassa_gomologiy}
    Consider the surface $S_{g, n}^{b}$ with a fixed hyperbolic metric $X$.
    Let $x_1, x_2 \in H_1(S_{g,n}^{b}, \mathbb{Z})$ be two strictly primitive homology classes. 
	Then there exist positive constants $C$ and $L_0$ such that for all $L \ge L_0$ the following inequalities hold:
    \[
        h_{g, n}^{b}\left(\frac{1}{C} L, x_1, X\right) \le h_{g, n}^{b}(L, x_2, X) \le h_{g, n}^{b}(C L, x_1, X).   
    \]
\end{proposition}
\begin{proof}
    Consider a mapping class $f \in \mathrm{Mod}_{g, n}^{b}$ such that $f(x_1) = x_2$.
    Then $h_{g, n}^{b}(L, x_1, f^{*}X) = h_{g, n}^{b}(L, x_2, X)$.
    To complete the proof, it remains to apply Proposition \ref{th_nezavisimost_ot_metriki}.
\end{proof}

\begin{proposition} \label{th_nezavisimost_ot_b}
    For any two hyperbolic metrics $X_1, X_2$ on the surfaces $S_{g, n}^{b}$ and $S_{g, n+b}$, respectively, and any
    pair of strictly primitive nonzero homology classes $x_1 \in H_{1}(S_{g, n}^{b}, \mathbb{Z})$ and $x_2 \in H_{1}(S_{g, n+b}, \mathbb{Z})$,
    there exists a constant $c > 0$
    such that for all sufficiently large $L$ the following inequality holds:
    \[ h_{g,n+b}\left(\frac{L}{c}, x_2, X_2\right) \le h_{g,n}^{b}(L, x_1, X_1) \le h_{g,n+b}(c L, x_2, X_2). \]
\end{proposition}
\begin{proof}
    Consider the special case. Let the metric $X_2$ be equal to $(X_1)_{\infty}$,
    where the construction of the metric $(X_1)_{\infty}$ is described in Construction \ref{con_Psi}.
    Then, for the proof it suffices to use Lemmas \ref{le_ocenka_general_Mirz} and \ref{le_svoystva_Psi}.  
        
    The general case reduces to the special one by means of Propositions \ref{th_ne_zavisimosti_assimptotik_ot_klassa_gomologiy} 
    and \ref{th_nezavisimost_ot_metriki}.
\end{proof}

Theorem \ref{th_nezavisimost_description} is a combination of Propositions \ref{th_nezavisimost_ot_metriki}, \ref{th_ne_zavisimosti_assimptotik_ot_klassa_gomologiy}, 
and \ref{th_nezavisimost_ot_b}.  
It follows that the leading order of growth of the function $h_{g,n}^{b}(L, x, X)$ does not depend on the metric $X$, is unchanged when replacing a geodesic boundary component by a puncture, and is unchanged when replacing the homology class $x$ by another strictly primitive class $x_1$. 
Thus, up to the leading order of growth, it suffices to study only the function $h_{g,n}(L, x)$, where $x$ is a nonzero strictly primitive 
homology class. From now on, we shall omit the metric symbol.

\subsection{General lower bound}
In this section we prove Theorem \ref{th_general_description}.

\begin{theorem} \label{th_S_g_n}
    Let a hyperbolic metric $X$ be fixed on the surface $S_{g,n}$ with $g+n \ge 3$ and $g \ge 1$. Then for 
    any strictly primitive homology class $x$ we have
   \[L^{6(g-1) + 2(n-1)} \ll h_{g,n}(L,x) \ll L^{6(g-1) + 2n}.\]
\end{theorem}
\begin{proof}
The upper bound is an immediate consequence of Mirzakhani's Theorem \ref{th_Mirz}. 
We now turn to the proof of the lower bound. Consider the case $n > 0$. 
Take disjoint simple closed curves 
$\alpha$ and $\alpha'$ such that $[\alpha'] = x$ and $\alpha \cup \alpha'$ cut off a subsurface 
homeomorphic to $S_{0, 1}^{2}$; see Figure \ref{fig_S_g_n}. 

\begin{figure}[h]
    \centering
    \begin{tikzpicture} 

		\draw (12.3,0.38) -- node[right] {$c$} cycle ;
		\filldraw[black] (12, 0) circle (2pt) ; 
		\filldraw[black] (2, 1) circle (2pt) ;
		\filldraw[black] (1.5, 0) circle (2pt) ;
		\filldraw[black] (2, -1) circle (2pt) ;
		
		\draw[ thick,  {stealth}-] (11.5,-0.1) to [out=90,in=180] (12,0.5)
		 to [out=0,in=90] (12.5,0) to [out=-90,in=0] (12,-0.5) to [out=180,in=-90] (11.5,0) ;

		\draw[ thick ] (11.5,0) to [out=90,in=180] (12,0.5)
		 to [out=0,in=90] (12.5,0) to [out=-90,in=0] (12,-0.5) to [out=180,in=-90] (11.5,0) ;

		\begin{scope}[shift={(3,0)}]
			\draw (7.7,-1.4) -- node[above] {$\alpha$} cycle ; 
			\draw (7.7,0.7) -- node[above] {$\alpha'$} cycle ;
			\draw[ thick] (7,2) to [out=-45, in=90] (7.38, 1.035);
			\draw[ thick,  -{stealth}] (7,2) to [out=-45,in = 90] (7.38, 1.035) ;
			\draw[ thick] (7, 0.07) to [out=45,in = -90] (7.38, 1.035) ;
			\draw[dashed] (7,2) to [out=-135,in=135] (7,0.07);

			\draw[thick] (7,-0.2) to [out=-45,in=90] (7.38, -1.035);
			\draw[thick,  -{stealth}] (7,-2) to [out=45,in = -90] (7.38, -1.035) ;
			\draw[thick] (7,-2) to [out=45,in = -90] (7.38, -1.035) ;
			\draw[dashed] (7,-0.2) to [out=-135,in=135] (7,-2);
		\end{scope}
        \draw[ultra thick] (1,0) to [out=90,in=180] (4,2) to [out=0,in=0] (7,2) to
        [out=0,in=90] (13,0) to [out=-90,in=0] (7,-2) to [out=180,in=-90] (1,0); 
        \draw[very thick] (3.5,0) to [out=-45,in=-135] (4.5, 0); 
        \draw[very thick] (3.6,-0.1) to [out=45,in=135] (4.4, -0.1); 
        \draw[very thick] (6.5,0) to [out=-45,in=-135] (7.5, 0); 
        \draw[very thick] (6.6,-0.1) to [out=45,in=135] (7.4, -0.1); 
        \draw[very thick] (9.5,0) to [out=-45,in=-135] (10.5, 0); 
        \draw[very thick] (9.6,-0.1) to [out=45,in=135] (10.4, -0.1); 
         
    \end{tikzpicture}
    \caption{Construction for Theorem \ref{th_S_g_n}.}
    \label{fig_S_g_n}
\end{figure}
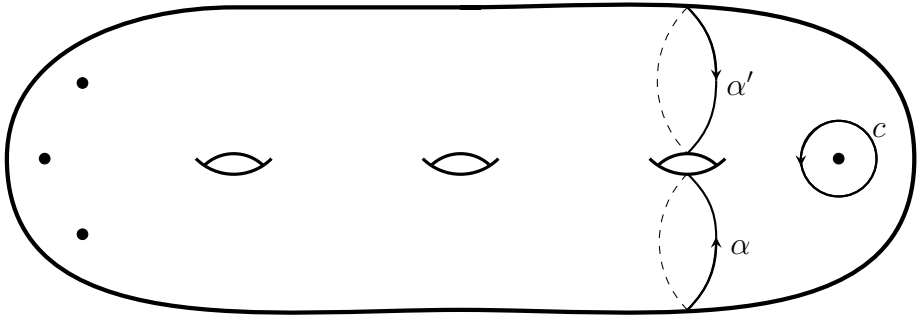
Let $\alpha$ and $\alpha'$ be a separating pair, and in the case of a surface 
with punctures, together with a distinguished puncture $c$, they form a subsurface homeomorphic to $S_{0,1}^{2}$.
Note that $[\alpha'] = [\alpha] + [c]$, since these geodesics together separate the surface.
Consider the auxiliary surface $S_{g-1,n}^{2}$, obtained from $S_{g,n}$ by cutting along $\alpha$.
Denote by $X'$ the metric on $S_{g-1,n}^{2}$ induced by the metric $X$ after cutting.
On $S_{g-1,n}^{2}$, the curve $\alpha'$ becomes separating. 
Any curve on $S_{g-1,n}^{2}$ of the same topological type as $\alpha'$ will have homology class $x$.

By Theorem \ref{th_general_Mirz} applied to $S_{g-1,n}^{2}$ with the metric $X'$, we have
\[s_{g-1,n}^{2}(L, \alpha') \asymp L^{6(g-1)+2(n-1)}.\]
Hence, on $S_{g,n}$ there exist at least $c L^{6(g-1)+2(n-1)}$ curves representing 
the class $x$ whose length does not exceed $L$. To complete the proof, it remains 
to use Theorem \ref{th_Den}, with which we will show that after gluing the surface,
the resulting homotopy classes of curves remain distinct. The Dehn--Thurston coordinates,
compatible with the set $\alpha \in \mathcal{P}$, are the same
for curves not intersecting $\alpha$ and lying on $S_{g,n}$ as for curves lying on $S_{g-1,n}^{2}$.
Consequently, we obtain the estimate
\[
    s_{g-1,n}^{2}(L, \alpha') \le h_{g,n}(L,x).
\]
The closed surface case is analogous; one considers a pair $\alpha \cup \alpha'$ cutting off 
a genus-one subsurface.
\end{proof}
    
\section{\texorpdfstring{Case $S_{1, 2}$}{Case S12}} \label{Case S_1_2}
    
    \subsection{Main result} 
    Recall that the function $h_{1, 2}(L, x)$ is the number of simple closed geodesics of length $\le L$ on $S_{1, 2}$ whose homology
    class is equal to $x$. We now state the main result of this section.

    \begin{theorem} \label{th_ocenka_S_1_2}
        Let a hyperbolic metric $X$ be fixed on the surface $S_{1,2}$. 
        Then for any fixed strictly primitive nonzero homology class $x \in H_1(S_{1,2}, \mathbb{Z})$, we have
        \[
            h_{1,2}(L, x) \gg L^{3.011057 \ldots}.
        \]
    \end{theorem}

    To prove Theorem \ref{th_ocenka_S_1_2}, we shall need several auxiliary results.

    We introduce the following notation.
    Let some Dehn--Thurston coordinates $(m, t)$ be fixed on the set $\mathcal{MC}_{0}^{4}$ 
    (see Definition in Section \ref{RefProject_D-T_coord}).
    Denote by $\gamma_{(m,t)}$ the integral multicurve with coordinates $(m, t)$.

    Let $\alpha_0, \alpha_1, \alpha_2, \alpha_3$ be the boundary components of the surface $S_{0}^4$.
    Let $\gamma$ be an essential simple closed curve on $S_{0}^4$ which is not isotopic to a boundary component.
    We say that $\gamma$ is of type $\Theta_i$ if it separates the boundary components $\alpha_0$ and
    $\alpha_i$ from the other boundary components of $\partial S_{0}^{4}$.
    It is easy to see that each of the three types contains some simple
    closed geodesic.
    Note that an integral multicurve $\gamma$ on $S_{0}^{4}$ can consist only
    of parallel copies of some simple closed curve $\delta$. The type of an integral multicurve $\gamma$
    is then defined to be the type of the curve $\delta$.

    \begin{theorem} \label{th_resenie_zadachi_S_0_4}
        The multicurve $\gamma_{(m, t)}$ consists of exactly one component if and only if \\ $\gcd\left(\frac{m}{2}, t\right) = 1$. 
        The partition into types is as follows.
        \begin{enumerate}
            \item $\Theta_1 = \{\gamma_{(m,t)} \mid (m, t) \in 2\mathbb{N} \times \mathbb{Z},\ \gcd\left(\frac{m}{2}, t\right) = 1,\ t - \text{ even}\},$ 
            \item $\Theta_2 = \{\gamma_{(m,t)} \mid (m, t) \in 2\mathbb{N} \times \mathbb{Z} \cup \{(0, 1)\},\ \gcd\left(\frac{m}{2}, t\right) = 1,\ t - \text{ odd},\ \frac{m}{2} - \text{ even}\},$
            \item $\Theta_3 = \{\gamma_{(m,t)} \mid (m, t) \in 2\mathbb{N} \times \mathbb{Z},\ \gcd\left(\frac{m}{2}, t\right) = 1,\ t - \text{ odd},\ \frac{m}{2} - \text{ odd}\}.$
        \end{enumerate}
        In general, the integral multicurve $\gamma_{(m,t)}$ consists of 
        $\Delta = \gcd\left(\frac{m}{2}, t\right)$ parallel 
        copies of the simple closed geodesic $\gamma_{\left( \frac{m}{\Delta}, \frac{t}{\Delta} \right)}$.
    \end{theorem}
    \begin{proof}
    Let $\varphi$ be the left Dehn twist along $\gamma_{(0,1)}$, and let $\psi$ be the left Dehn twist 
    along the curve $\gamma_{(2,0)}$.
    From Lemma \ref{le_deystvie_T_gamma_0_1} we have the following formulas
    \begin{align*}
     \varphi^{\pm1}\bigl(\gamma_{(m, t)}\bigr) =& ~\gamma_{(m, t \mp m)} \\
     \psi^{\pm1}\bigl(\gamma_{(m, t)}\bigr) =& ~\gamma_{\big(|m \pm 4t|, t \cdot \theta(m \pm 4t)\big)}. 
    \end{align*}
    Note that the maps $\varphi^{\pm 1}$ and $\psi^{\pm 1}$ do not change the type of an integral multicurve.

    Consider an arbitrary integral multicurve $\gamma_{(m,t)}$. 
    We adopt the convention that $\gamma_{(m,t)} = \gamma_{(-m,-t)}$.
    Then the action of $\psi$ can be written as
    \[ \psi^{\pm 1}(\gamma_{(m, t)}) = \gamma_{(m \mp 4t, t)}.\]
    By Lemma \ref{le_o_harakterizacii_multicrivyh}, it suffices to consider the case 
    where $\frac{m}{2}$ and $t$ are coprime. We now describe an analogue of the Euclidean algorithm, 
    in which we modify the integral multicurve, preserving its type and decreasing the parameter 
    $k = |m| + |t|$.

    \begin{enumerate}
    \item If $\frac{|m|}{2} < |t|$ and $\frac{|m|}{2} \neq 0$, then using $\varphi^{\pm 1}$ we pass to the geodesic 
    $\gamma_{(m, \tilde{t})}$ for which $\frac{|m|}{2} \ge |\tilde{t}|$. Note that the type 
    of the new geodesic $\gamma_{(m, \tilde{t})}$ is the same, and the inequality
    \[
    \tilde{k} = |m| + |\tilde{t}| \le |m| + \frac{|m|}{2} < |m| + |t| = k
    \]
    holds. Thus $\tilde{k} < k$, and the parameters of the new curve are again coprime.
    \item If $\frac{|m|}{2} = |t|$, then from $\gcd\left(\frac{m}{2}, t\right) = 1$ it follows 
    that $(m,t) \in \{\pm (2, 1), \pm(2, -1)\}$. Using $\varphi^{\pm 1}$ and the convention $(m,t) = (-m, -t)$, 
    we pass to the geodesic $\gamma_{(2, 1)}$, which is of type $\Theta_3$.
    \item If $\frac{|m|}{2} > |t|$ and $|t| \neq 0$, then using the maps $\psi^{\pm}$ we pass
    to the geodesic $\gamma_{(\tilde{m}, t)}$ for which $|\tilde{m}| < m$.
    Thus $\tilde{k} = |\tilde{m}| + |t| < k$.
    \item If $\frac{|m|}{2} = 0$, then from $\gcd\left(\frac{m}{2}, t\right) = 1$ it follows 
    that $(m,t) \in \{\pm (0, 1), \pm(0, -1)\}$. Using the convention $(m,t) = (-m, -t)$, 
    we pass to the geodesic $\gamma_{(0, 1)}$, which is of type $\Theta_2$.
    \item If $|t| = 0$, then from $\gcd\left(\frac{m}{2}, t\right) = 1$ it follows 
    that $(m,t) \in \{\pm (2, 0)\}$. Hence we pass to the geodesic 
    $\gamma_{(2, 0)}$, which is of type $\Theta_1$. 
    \end{enumerate}

    Thus, the above algorithm reduces any simple closed geodesic $\gamma_{(m,t)}$, 
    without changing its type, to one of $\gamma_{(0,1)}, \gamma_{(2,1)}, \gamma_{(2,0)}$.
     Note that the transformations 
    $\varphi^{\pm 1}$ and $\psi^{\pm}$ used above preserve $\gcd\left(\frac{m}{2}, t\right)$,
    the parity of $\frac{m}{2}$, and the parity of $t$.
    \end{proof}

    \begin{remark}
       The theorem proved above holds for any topological type of surface $S_{0, a}^{b}$ satisfying $a + b = 4$.
    \end{remark}

    \subsection{Some facts from hyperbolic geometry}

    \begin{lemma}[see \cite{litlink2}] \label{th_shest}
    Let a right-angled hyperbolic hexagon be given (see Figure \ref{fig3}). Then all perpendiculars to 
    opposite sides lie inside the hexagon, and the following metric relations hold:
    \[
        \sinh b_2 \, \sinh a_3 = \cosh h;
    \]
    \begin{equation} \label{eq_chb_1}
        \cosh b_1 \, \sinh a_2 \, \sinh a_3 = \cosh a_1 + \cosh a_2 \, \cosh a_3.
    \end{equation}
    In particular, the hexagon is uniquely determined by its three pairwise non-adjacent sides.
    \end{lemma}

    \begin{figure}[h]
        \centering
        \begin{tikzpicture} 
            \draw (0, 0) -- node[left] {$h$} cycle ;
            \draw (0, 1.7) -- node[above] {$a_1$} cycle ; 
            \draw (0, -1.7) -- node[below] {$b_1$} cycle ;
            \draw (1.9, 0.7) -- node[above] {$b_2$} cycle ;
            \draw (-1.9, 0.7) -- node[above] {$b_3$} cycle ;
            \draw (1.9, -0.7) -- node[below] {$a_3$} cycle ;
            \draw (-1.9, -0.7) -- node[below] {$a_2$} cycle ;
            \draw[very thick] (2,0) to [out=135,in=-75] (1,1.732) 
                                    to [out=-165,in=-25] (-1,1.732)
                                    to [out=-115,in=45] (-2,0)
                                    to [out=-45,in=105] (-1,-1.732)
                                    to [out=15,in=165] (1,-1.732)
                                    to [out=75,in=-135] (2.01,0.01);
             \draw[very thick] (1.8,-0.2) to [out=135,in=-45] (1.6,0)  to [out=45,in=-135] (1.8,0.2);
             \begin{scope}[rotate=60]
             \draw[very thick] (1.8,-0.2) to [out=135,in=-45] (1.6,0)  to [out=45,in=-135] (1.8,0.2);
             \end{scope}
             \begin{scope}[rotate=120]
             \draw[very thick] (1.8,-0.2) to [out=135,in=-45] (1.6,0)  to [out=45,in=-135] (1.8,0.2);
             \end{scope}
             \begin{scope}[rotate=180]
             \draw[very thick] (1.8,-0.2) to [out=135,in=-45] (1.6,0)  to [out=45,in=-135] (1.8,0.2);
             \end{scope}
              \begin{scope}[rotate=-60]
             \draw[very thick] (1.8,-0.2) to [out=135,in=-45] (1.6,0)  to [out=45,in=-135] (1.8,0.2);
             \end{scope}
             \begin{scope}[rotate=-120]
             \draw[very thick] (1.8,-0.2) to [out=135,in=-45] (1.6,0)  to [out=45,in=-135] (1.8,0.2);
             \end{scope}
            \draw[very thick] (0,1.55) to [out = -90, in = 90] (0,-1.6);
            \draw[very thick] (0.2,1.55) to [out=-90,in=90] (0.2, 1.35) to [out=180,in=0] (0, 1.35);
            \draw[very thick] (0.2,-1.57) to [out=90,in=-90] (0.2, -1.35) to [out=180,in=0] (0, -1.35);
    
        \end{tikzpicture}
        \caption{A hyperbolic hexagon.}
        \label{fig3}
    \end{figure}
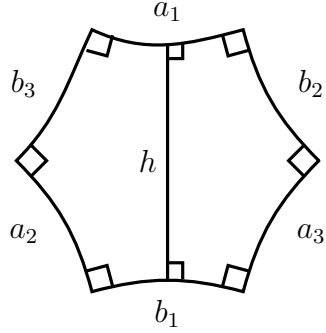

    \begin{lemma} \label{le_dlina_perpendikulyara_v_shest}
    Let a right-angled hyperbolic hexagon be given (see Figure \ref{fig3}). Then the following equality holds:
    \begin{align}
        \cosh 2h = 2 \frac{\cosh^2 b_2 + \cosh^2 b_3 + 2 \cosh b_1 \cosh b_2 \cosh b_3}{\sinh^2 b_1} + 1. \label{eq_ch2h}
    \end{align}
    \end{lemma}
    \begin{proof}
    Using the equalities from Lemma \ref{th_shest}, we obtain
    \begin{align}
        \cosh 2h &= 2 \cosh^2 h - 1 = \sinh^2 b_2 \sinh^2 a_3 + \sinh^2 b_3 \sinh^2 a_2 - 1 \\
                &= \sinh^2 b_2 \cosh^2 a_3 + \sinh^2 b_3 \cosh^2 a_2 - \sinh^2 b_2 - \sinh^2 b_3 - 1. \label{eq_sh_2b_2}
    \end{align}
    We use equality \eqref{eq_chb_1} from Lemma \ref{th_shest}. 
    \begin{align*}
        \sinh^2 b_2 \cosh^2 a_3 
        &= \left( \frac{\cosh b_3 + \cosh b_1 \cosh b_2}{\sinh b_1} \right)^2 \\
        &= \frac{\cosh^2 b_2 + \cosh^2 b_3 + 2\cosh b_1 \cosh b_2 \cosh b_3}{\sinh^2 b_1} + \cosh^2 b_2.
    \end{align*}
    Substituting the obtained expression for $\sinh^2 b_2 \cosh^2 a_3$ and the analogous expression for 
    $\sinh^2 b_3 \cosh^2 a_2$ into formula \eqref{eq_sh_2b_2}, we obtain the required equality.
    \end{proof}

    \begin{lemma} \label{le_assimptotika_perpendikulyara_v_shest}
    Let a right-angled hyperbolic hexagon be given (see Figure \ref{fig3}) and suppose that $b_1 \le b_2$. Then the following estimate holds:
    \[
        2h \le 2(b_2 - b_1) + F(b_1) + G(b_3),
    \]
    where $F(b_1)$ is positive and strictly decreasing, and $G(b_3)$ is positive and strictly increasing.

    \end{lemma}
    \begin{proof}
    From equality \eqref{eq_ch2h} and the inequality $b_1 \le b_2$, it follows that
    \[
        \cosh 2h \le 3 \frac{\cosh^2 b_2 + \cosh^2 b_3 + 2 \cosh b_1 \cosh b_2 \cosh b_3}{\sinh^2 b_1}.
    \]
    Using the elementary estimate $\frac{e^x}{2} \le \cosh x \le e^x$, we obtain 
    \[
        e^{2h} \le 6 \frac{e^{2b_2} + e^{2b_3} + 2 e^{b_1 + b_2 + b_3}}{\sinh^2 b_1} \le 
        6 e^{2b_2} \frac{1 + e^{2b_3} + 2 e^{b_3}}{\sinh^2 b_1}.
    \]
    Taking logarithms, we obtain
    \[
        2h \le 2b_2 - 2\log(\sinh b_1) + \log(1 + e^{2b_3} + 2 e^{b_3}) + \log 6.
    \]
    After simplification, we get
    \[
        2h \le 2(b_2 - b_1) + F(b_1) + G(b_3),
    \]
    where $F(b_1) = 2 \log\left(\frac{2}{1 - e^{-2 b_1}}\right)$ and $G(b_3) = 2 \log(1 + e^{b_3}) + \log 6$.
    It remains to note that the functions $F$ and $G$ satisfy the conditions of the lemma.
    \end{proof}

    Let $P$ be a hyperbolic pair of pants whose boundary components are geodesic.
    Denote the boundary components of $P$ by $\gamma_0, \gamma_1, \delta$. 
    Cut $P$ into two right-angled hyperbolic hexagons $P_i$, $i \in \{1, 2\}$,
    by means of geodesic arcs $a, b, c$ lying in $P$, whose endpoints lie on $\gamma_1$ and $\delta$, 
    $\gamma_0$ and $\gamma_1$, and $\gamma_0$ and $\delta$, respectively, so that each of these arcs is orthogonal to the corresponding boundary components.
    Let $\gamma_0^i$ denote the arc obtained by intersecting the geodesic $\gamma_0$ with the hexagon
    $P_i$. Draw the perpendiculars $h_i$ in the hexagons $P_i$ between $\gamma_0^i$ and $a$.
    Note that the figure is symmetric with respect to the horizontal axis; in particular, $\ell_{h_1} = \ell_{h_2}$.
    The above construction is illustrated in Figure \ref{fig_10}.
    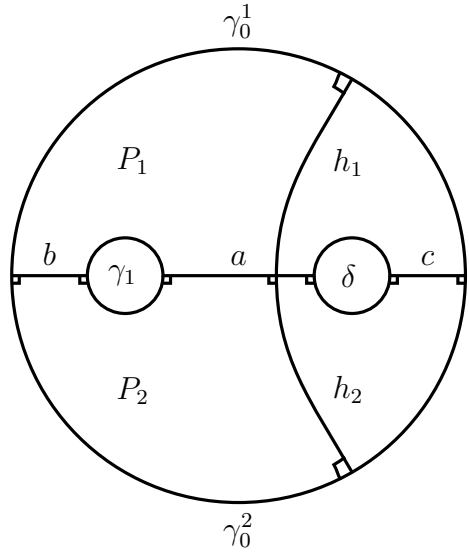
\begin{figure}[h]

        \centering
        
        \begin{tikzpicture} 
            \draw (-1.2, 0) -- node[left] {$\gamma_1$} cycle;
            \draw ( 1.7, 0) -- node[left] {$\delta$} cycle ;
            \draw (-2.5, 0) -- node[above] {$b$} cycle ; 
            \draw (2.5, 0) -- node[above] {$c$} cycle ;
            \draw (0, 0) -- node[above] {$a$} cycle ;
            \draw (1.1, 1.5) -- node[right] {$h_1$} cycle ;
            \draw (1.1, -1.5) -- node[right] {$h_2$} cycle ;
            \draw (0, 3) -- node[above] {$\gamma_0^1$} cycle ;
            \draw (0, -3) -- node[below] {$\gamma_0^2$} cycle ;
            \draw (-1.4, 1.2) -- node[above] {$P_1$} cycle ;
            \draw (-1.4, -1.2) -- node[below] {$P_2$} cycle ;

            \draw[very thick] (-3,0) to [out=0, in=180] (-2,0) 
                                    to [out=90, in=180] (-1.5,0.5)
                                    to [out=0, in=90] (-1,0)
                                    to [out=0, in=180] (1,0)
                                    to [out=90, in=180] (1.5,0.5)
                                    to [out=0,in=90] (2,0)
                                    to [out=0,in=180] (3,0)
                                    to [out=90,in= 0] (0,3)
                                    to [out=180,in= 90] (-3,0)
                                    to [out=-90,in= 180] (0,-3)
                                    to [out=0,in=-90] (3,0);

            \draw[very thick] (1.5, 2.598) to [out=-120, in=90] (0.5,0)
                                            to [out=-90, in=120] (1.5, -2.598); 

            \draw[very thick] (-2,0) to [out=-90, in=180] (-1.5,-0.5)
                                        to [out=0, in=-90] (-1,0);

            \draw[very thick] (2,0) to [out=-90, in=0] (1.5,-0.5)
                                        to [out=180, in=-90] (1,0);

            \draw[very thick] (-2.9, 0) to [out=-90, in=90] (-2.9,-0.1)
            to [out=180, in=0] (-3,-0.1);

            \draw[very thick] (-0.9, 0) to [out=-90, in=90] (-0.9,-0.1)
            to [out=180, in=0] (-1,-0.1);
            
            \draw[very thick] (2.1, 0) to [out=-90, in=90] (2.1,-0.1)
            to [out=180, in=0] (2,-0.1);

            \draw[very thick] (-2.1, 0) to [out=-90, in=90] (-2.1,-0.1)
            to [out=0, in=180] (-2,-0.1);

            \draw[very thick] (0.4,0) to [out=-90, in=90] (0.4,-0.1)
            to [out=0, in=180] (0.5,-0.1);

            \draw[very thick] (2.9,0) to [out=-90, in=90] (2.9,-0.1)
            to [out=0, in=180] (3,-0.1);

            \draw[very thick] (0.9,0) to [out=-90, in=90] (0.9,-0.1)
            to [out=0, in=180] (1,-0.1);

            \draw[very thick] (1.35, 2.698) to [out=-120, in=60] (1.26,2.5)
            to [out=-30, in=145] (1.41,2.4);

            \draw[very thick] (1.35, -2.698) to [out=120, in=-60] (1.26, -2.5)
            to [out=30, in= -145] (1.41, -2.4);

        \end{tikzpicture}
        
        \caption{Figure for Lemma \ref{le_poluchim_b_le_L_l_alpha}.}
        
        \label{fig_10}
    \end{figure}
    From Lemma \ref{le_assimptotika_perpendikulyara_v_shest} we immediately obtain the following.

    \begin{lemma}\label{le_poluchim_b_le_L_l_alpha}
        Suppose that $\ell_{\gamma_0} \le \ell_{\gamma_1}$. Then
        \[
            2 \ell_{h_1} = \ell_{h_1} + \ell_{h_2} \le \ell_{\gamma_1} - \ell_{\gamma_0} + F(\ell_{\gamma_0}) + G(\ell_{\delta}), 
        \]
        where $F(x) > 0$ is strictly decreasing and $G(x) > 0$ is strictly increasing.
    \end{lemma}

    We now consider the surface $S_{1, 2}$ equipped with a fixed hyperbolic metric $X$.
    Fix two disjoint and non-isotopic simple closed non-separating geodesics $\gamma_0$ and $\gamma_1$, 
    each of which is not isotopic to a loop around a puncture. Denote by $\Sigma_{\gamma_1}$ the hyperbolic 
    surface obtained from $S_{1, 2}$ by cutting along $\gamma_1$. Here (and below) we mean 
    that $\Sigma_{\gamma_1}$ is a hyperbolic surface of topological type $S_{0, 2}^{2}$ with 
    boundary consisting of two copies of $\gamma_1$.

    \begin{lemma} \label{th_conbinatorial_lenght_S_0_4}
    Suppose that $\ell_{\gamma_0} \le \ell_{\gamma_1}$. Then for integral multicurves 
    on $\Sigma_{\gamma_1}$ there exist 
    Dehn--Thurston coordinates $(m,t)$ associated with the pants decomposition $\mathcal{P} = \{\gamma_0\}$, 
    satisfying the following properties:
            
    \begin{enumerate}
        \item For any integral multicurve $\gamma = \gamma_{(m, t)}$ on $\Sigma_{\gamma_1}$
        with Dehn--Thurston coordinates $(m,t)$, the following inequality holds:
        \[
        \ell_{\gamma} \le m \ell_{\gamma_1} + |t| \ell_{\gamma_0} + m C, 
        \] 
        where $C$ is a positive constant depending only on the metric $X$.
        \item The curve $\gamma_{(2, 0)}$ is separating on $S_{1,2}$.
    \end{enumerate}

    \end{lemma}
    
    \begin{proof}
    Assume that the Dehn--Thurston coordinates have already been chosen. 
    Then we use system \eqref{eq_SR} from Lemma \ref{le_deystvie_T_gamma_0_1} for the case $m > 0$:
    \begin{equation*} 
    \gamma_{(m, t)} = \begin{cases}
        \mathrm{SR}\left(t \gamma_{0}, \frac{m}{2} \gamma_{(2, 0)}\right), & t > 0;\\
        \mathrm{SR}\left(\frac{m}{2} \gamma_{(2, 0)}, |t| \gamma_{0}\right), & t < 0.
    \end{cases}
    \end{equation*}
    From the definition of $\mathrm{SR}$ it follows immediately that
    \begin{equation}\label{eq_alpha}
        \ell_{\gamma_{(m, t)}} \le |t|\ell_{\gamma_{0}} + \frac{m}{2} \ell_{\gamma_{(2, 0)}}. 
    \end{equation}  
    If $m = 0$ or $t = 0$, the inequality also holds and is verified directly.

        The remainder of the proof proceeds as follows.
        We construct a simple closed geodesic 
        $\alpha$ on $\Sigma_{\gamma_1}$ satisfying two conditions:
        $\alpha$ separates the surface $S_{1,2}$ and $i(\alpha, \gamma_0) = 2$. Then one can 
        choose the required coordinate system so that $\alpha = \gamma_{(2, 0)}$. 
        Thus, to complete the proof of the theorem, it remains to construct the geodesic $\alpha$, 
        estimate its length, and substitute the resulting estimate into inequality \eqref{eq_alpha}.

        We now turn to the construction of the geodesic $\alpha$. Let $P$ be one of the two pairs of pants
        corresponding to the decomposition $\mathcal{P} = \{\gamma_0\}$. 
        In the notation of Lemma \ref{le_poluchim_b_le_L_l_alpha}, there exists on $P$ a 
        geodesic segment $h = h_1 \cup h_2$ with endpoints on $\gamma_0$ such that 
        the following estimate holds:
        \[
            \ell_h \le \ell_{\gamma_1} - \ell_{\gamma_0} + F(\ell_{\gamma_0}) + G(0), 
        \]
        where $F(x) > 0$ is strictly decreasing.

        We construct an analogous geodesic segment for the second pair of pants and denote it by $\hat{h}$.
        Now join the endpoints of $h$ and $\hat{h}$ by arcs lying in a small neighbourhood of $\gamma_0$
        so that the resulting closed polygonal curve separates the surface $S_{1,2}$. It is easy to check that 
        it suffices to use arcs whose total length is $\le \ell_{\gamma_0}$.
        Thus, we obtain a polygonal curve whose length is $\le \ell_h + \ell_{\hat{h}} + \ell_{\gamma_0}$.
        Let $\alpha$ be the geodesic representative of the homotopy class of the constructed polygonal curve.
        Hence, the following inequality holds:
        \[\ell_{\alpha} 
            \le 2\ell_{\gamma_1} - \ell_{\gamma_0} + 2\left(F(\operatorname{sys}_{S_{1, 2}}) + G(0)\right),\]
        where $\operatorname{sys}_{S_{1, 2}}$ denotes the length of the shortest simple closed geodesic on $S_{1,2}$.
        To complete the proof, it remains to substitute this inequality into \eqref{eq_alpha}.
    \end{proof}
    
    \begin{remark}
    The statement of the theorem is substantially weakened. One could obtain the inequality
    \[
        \ell_{\gamma_{(m, t)}} \le m \ell_{\gamma_1} + \left( |t| - \frac{m}{2} \right) \ell_{\gamma_0} + C m.
    \]
    However, we have deliberately weakened it to simplify the subsequent computations.
    \end{remark}

    Let $\delta$ denote an oriented loop around a puncture on the surface $S_{1,2}$.
    In the notation of Lemma \ref{th_conbinatorial_lenght_S_0_4}, we formulate a statement 
    that combines the results obtained above. This lemma will be used in the proof of Theorem \ref{th_ocenka_S_1_2}.
    
    \begin{lemma}[Special Dehn--Thurston coordinates]\label{le_specialnye_coordinaty_D_T}
    On the set of integral multicurves lying in $\Sigma_{\gamma_1}$, there exist Dehn--Thurston coordinates $(m,t)$
    associated with the pants decomposition $\mathcal{P} = \{\gamma_0\}$, satisfying the following properties:
        \begin{enumerate}
        \item $(m, t) \in 2\mathbb{N} \times \mathbb{Z} \cup \{0\} \times \mathbb{N}$.
        \item If $\ell_{\gamma_0} \le \ell_{\gamma_1}$, then 
            \begin{align}\label{al_specialnye_metricheskoe_neravenstvo}
                \ell_{\gamma_{(m,t)}} \le m \ell_{\gamma_1} + |t| \ell_{\gamma_0} + m C,
            \end{align}
            where $C$ is a positive constant depending only on the metric $X$ on $S_{1,2}$.
        \item The multicurve $\gamma_{(m,t)}$ consists of $\Delta = \gcd\left(\frac{m}{2}, t\right)$ parallel copies of the simple geodesic $\gamma_{\left( \frac{m}{\Delta}, \frac{t}{\Delta} \right)}$.
        \item If the integral multicurve $\gamma_{(m,t)}$ is a simple geodesic and $t$ is even, then $\gamma_{(m,t)}$ separates $S_{1,2}$.
        \item Let the integral multicurve $\gamma_{(m,t)}$ be a simple geodesic and let $t$ be odd.
            Fix an orientation on the geodesics $\gamma_1$ and $\gamma_{(m,t)}$ so that $\gamma_1$ and $\gamma_{(m,t)}$ belong to the same class in $H_1(S_{1,2}, \mathbb{Z}) / \langle [\delta] \rangle$.
            Then 
        \begin{equation}\label{eq_gomology_formula}
            [\gamma_{(m,t)}] = [\gamma_1] + (-1)^{\frac{m}{2}} \sigma(\gamma_1) [\delta],
        \end{equation}
        where $\sigma(\gamma_1)$ takes values in $\{-1, 1\}$ and changes sign when the orientation of $\gamma_1$ is reversed.
        \end{enumerate}    
    \end{lemma}
    \begin{proof}
        The first assertion follows from the construction of Dehn--Thurston coordinates (see Section \ref{RefProject_D-T_coord}).
        Fix on the surface $\Sigma_{\gamma_1}$ the coordinates from Lemma \ref{th_conbinatorial_lenght_S_0_4}.
        Then the second assertion holds, and the geodesic $\gamma_{(2, 0)}$ separates the surface 
        $S_{1, 2}$.
        The third assertion follows from Lemma \ref{le_o_harakterizacii_multicrivyh}.
        The fourth and fifth assertions follow from Theorem \ref{th_resenie_zadachi_S_0_4} and from the fact that 
        the geodesic $\gamma_{(2, 0)}$ separates the surface $S_{1, 2}$.
    \end{proof}

   \subsection{Combinatorial problem}
    Let $A(m, t)$ denote the matrix
    \begin{equation*}  
        \begin{pmatrix}
            m & 1 & 0 \\
            |t| & 0 & 0 \\
            m & 0 & 1
        \end{pmatrix}.
    \end{equation*}

   \begin{theorem}[Solution of the combinatorial problem] \label{th_reshenie_combinatornoy_zadachi}
    Let $\mathcal{Z}_n(L)$ denote the number of tuples $(m_1, t_1, \ldots, m_{2n}, t_{2n})$ such that:
    \begin{enumerate}
        \item $t_i$ are odd integers;
        \item $m_i$ are even positive integers;
        \item For any $i$, $\gcd(m_i / 2, t_i) = 1$;
        \item The inequality $\| A(m_1, t_1) \cdots A(m_{2n}, t_{2n}) \| \le L$ holds,
            where the norm is the operator $L^1$-norm $\|A\| = \max_{1 \le j \le n} \sum_{i=1}^{n} |a_{ij}|$.
        \item For the numbers $m_i / 2$, define residues $q_i \in \mathbb{Z} / 2\mathbb{Z}$ as follows:
            \begin{equation*}
            \begin{pmatrix}
                q_1 \\
                q_2 \\
                q_3 \\
                \vdots \\
                q_{2n}
            \end{pmatrix} \equiv 
            \begin{pmatrix}
                1 & 0 & 0 & \ldots & 0 \\
                1 & 1 & 0 & \ldots & 0 \\
                1 & 1 & 1 & \ldots & 0 \\
                \vdots & \vdots & \vdots & \ddots & \vdots \\
                1 & 1 & 1 & \ldots & 1
            \end{pmatrix}
            \begin{pmatrix}
                \frac{m_1}{2} + 1 \\
                \frac{m_2}{2} + 1 \\
                \frac{m_3}{2} + 1 \\
                \vdots \\
                \frac{m_{2n}}{2} + 1
            \end{pmatrix}
            \pmod{2},
        \end{equation*}
        then among the residues $q_i$ there must be exactly $n$ zeros and $n$ ones.
    \end{enumerate}
            
        Then the following lower bound holds:
        \[
        \sum_{n=1}^{\infty} \mathcal{Z}_n(L) \gg L^{3.011057 \ldots}.
        \]
    \end{theorem}
    To prove Theorem \ref{th_reshenie_combinatornoy_zadachi}, we formulate and prove a series of lemmas.

    Recall that the Möbius function is the function $\mu: \mathbb{N} \to \mathbb{Z}$ defined by
    \begin{equation*}
        \mu(d) = 
        \begin{cases}
            1 & d = 1, \\
            (-1)^r & d = p_1 \cdots p_r, \\
            0 & p^2 \mid d.
        \end{cases}
    \end{equation*}
    Its fundamental property is the identity
    \begin{equation*}
        \sum_{d \mid n} \mu(d) = 
        \begin{cases}
            1 & \text{if } n = 1, \\
            0 & \text{if } n \neq 1.
        \end{cases}
    \end{equation*}

    We shall need the following standard lemmas.

    \begin{lemma} \label{le_chislo_vzaimnoprostyh}
        Let $t \in \mathbb{N}$. Then
        \[
            \# \{m \in \mathbb{N} \mid \gcd(m, t) = 1, \ m \le x\} = \sum_{d \mid t} \mu(d) \left\lfloor \frac{x}{d} \right\rfloor.
        \]  
    \end{lemma}

    \begin{lemma} \label{le_nechet_zeta} 
        \[
            \sum_{\substack{d -\text{ odd}}} \frac{\mu(d)}{d^2} = \frac{8}{\pi^2}.
        \]
    \end{lemma}
    \begin{proof}
        We have
        \[
            \sum_{d} \frac{\mu(d)}{d^2}
            = \sum_{\substack{d -\text{ odd}}} \frac{\mu(d)}{d^2} + \sum_{\substack{d -\text{ even}}} \frac{\mu(d)}{d^2}
            = \sum_{\substack{d -\text{ odd}}} \frac{\mu(d)}{d^2} + \sum_{\substack{d -\text{ odd}}} \frac{\mu(2d)}{4d^2}
            = \left(1 - \frac{1}{4}\right) \sum_{\substack{d - \text{ odd}}} \frac{\mu(d)}{d^2}.
        \]
        To complete the proof, it suffices to use the identity
        \[
            \sum_{d} \frac{\mu(d)}{d^2} = \frac{6}{\pi^2}.
        \]
    \end{proof}
   In what follows, for sums with non-integer upper limit, we use the convention
    \[
    \sum_{k = 1}^{x} a_k = \sum_{k = 1}^{\lfloor x \rfloor} a_k.
    \]

    \begin{lemma}
        Let $\mathcal{X}_{1}(R, a)$ denote the number of pairs $(n, t)$ satisfying:
        \begin{enumerate} \label{en_usloviya}
            \item $t$ is an odd integer;
            \item $n$ is a positive integer;
            \item $\gcd(n, t) = 1$;
            \item the inequality $a n + |t| \le R$ holds.
        \end{enumerate}
        Then for any $a \ge 1$ and any $R \ge a + 1$, the following estimate holds:
        \begin{equation}\label{eq_neravenstvo_XLa}
            \left| \mathcal{X}_{1}(R, a) - \frac{4(R^2 - a^2)}{a \pi^2} \right|
            \le \frac{1 + 2a - 6a^2}{4a}
            + \frac{3(1 + a)}{2a} R
            + \frac{(a + 1)R - a^2}{2a} \log (R - a).
        \end{equation}
         
    \end{lemma}

    \begin{proof}
    Using Lemma \ref{le_chislo_vzaimnoprostyh}, we obtain
    \[
        \mathcal{X}_{1}(R, a) = 2 \sum_{\substack{t = 1 \\ t -\text{ odd}}}^{R - a} \sum_{d \mid t} 
        \mu(d) \left\lfloor \frac{R - t}{ad} \right\rfloor
        =
        2 \sum_{\substack{d = 1 \\ d -\text{ odd}}}^{R - a} \mu(d) \left( \sum_{\substack{s = 1 
        \\ s -\text{ odd}}}^{\frac{R - a}{d}} \left\lfloor \frac{R - sd}{ad} \right\rfloor \right).
    \]

    Using the inequality $\left| \lfloor x \rfloor - x + \frac{1}{2}  \right| \le \frac{1}{2}$, 
    we obtain the following estimate for the inner sum:
    \[
    \left| \sum_{\substack{s = 1 \\ s - \text{ odd}}}^{\frac{R - a}{d}} \left\lfloor \frac{R - sd}{ad} \right\rfloor - \frac{R^2 - a^2}{4ad^2} + \frac{R - a}{4d} \right|
    \le \frac{1 + 2a}{4a} + \frac{(a + 2)R - a^2}{4ad}.
    \]

    We complete the proof using Lemma \ref{le_nechet_zeta}, the condition $R \ge a + 1$, and the inequalities
    \begin{gather*}
    \left| \sum_{\substack{d = 1 \\ d -\text{ odd}}}^{R - a} \mu(d) \right| \le \frac{R - a + 1}{2}, \\
    \left| \sum_{\substack{d = 1 \\ d -\text{ odd}}}^{R - a} \frac{\mu(d)}{d} \right| \le 1 + \frac{1}{2} \log (R - a), \\
    \left| \sum_{\substack{d > R - a \\ d -\text{ odd}}}^{\infty} \frac{\mu(d)}{d^2} \right| \le \frac{1}{2(R - a)}.
    \end{gather*}
    These inequalities follow immediately from the bound $|\mu(d)| \le 1$.
\end{proof}

In exactly the same manner, the following lemma is derived.

\begin{lemma}
    Let $\mathcal{X}_{2}(R, a)$ denote the number of pairs $(n, t)$ such that:
    \begin{enumerate}
        \item The pair $(n, t)$ satisfies conditions 1--4 of Lemma \ref{en_usloviya}.
        \item $n \equiv 1 \pmod{2}$.
    \end{enumerate}
    Then for any $a \ge 1$ and any $R \ge a + 1$, the following estimate holds:
    \begin{equation}\label{eq_neravenstvo_X_2_La}
        \left| \mathcal{X}_{2}(R, a) - \frac{2(R^2 - a^2)}{a \pi^2} \right|
        \le \frac{3 - 6a}{8}
        + \frac{6a + 5}{8a} R
        + \frac{(a + 1)R - a^2}{4a} \log (R - a).
    \end{equation}
\end{lemma}

Fix two residues $\rho_1, \rho_2 \in \mathbb{Z} / 2\mathbb{Z}$.
\begin{lemma} \label{le_assimptotika_mu_i}
    Let $\lambda_k = \lambda_k(\rho_1, \rho_2)$ be the integers of the form $m_2(2m_1 + |t_1| + 1) + |t_2|$, arranged in non-decreasing order with multiplicities, where
    \begin{enumerate}
        \item $t_1, t_2$ are odd integers;
        \item $m_1, m_2$ are even positive integers;
        \item $\gcd(m_1 / 2, t_1) = \gcd(m_2 / 2, t_2) = 1$;
        \item $\frac{m_1}{2} \equiv \rho_1 \pmod 2$ and $\frac{m_2}{2} \equiv \rho_2 \pmod 2$.
    \end{enumerate}
    Then for any $M \ge 1000$ and any $k \ge M^3 r(M)$, the following inequality holds
    \[
        \lambda_k \le \sqrt[3]{\frac{k}{r(M)}},
    \]
    where
    \begin{equation} \label{eq_c_L}
        r(M) = \frac{1}{M^3} \left( \frac{(M - 5)^3 - 36^3}{96 \pi^2}
        - \frac{9}{128} \left( \log\left( \frac{M - 5}{4} \right) + 3 \right)
        \left( (M - 3)^2 - 36^2 \right) \right).
    \end{equation}

\end{lemma}
    \begin{proof} 
    We estimate from below the function $\#\{k \in \mathbb{N} \mid \lambda_k \le L\}$.

    Consider the case $\rho_1 = 0$ and $\rho_2 = 0$. We count only those $\lambda_k$ 
    for which $m_2 = 4$. In this case, the condition $\gcd(m_2 / 2, t_2) = 1$
    holds for any odd $t_2$, and the inequality
    $\lambda_k \le L$ becomes
    \[
        8n + |t_1| \le \frac{L - |t_2|}{4} - 1.
    \]
    Hence, for a fixed $t_2$, the number of pairs $(m_1, t_1)$ such that $\lambda_k \le L$
    is equal to $\mathcal{X}_{1}\left(\frac{L - |t_2|}{4} - 1, \, 8\right)$. 
    Thus, the following estimate holds:
    \begin{equation} \label{eq_lambda_k_le}
        \#\{k \in \mathbb{N} \mid \lambda_k \le L\}
        \ge 2 \sum_{\substack{t_2 = 1 \\ t_2- \text{ odd}}}^{L}
        \mathcal{X}_{1}\left(\frac{L - t_2}{4} - 1, \, 8 \right).
    \end{equation}

    In the case $\rho_1 = 0$ and $\rho_2 = 1$, the analogous estimate, obtained by counting only those 
    $\lambda_k$ for which $m_2 = 2$, is stronger than \eqref{eq_lambda_k_le}.
    Therefore, estimate \eqref{eq_lambda_k_le} remains valid.

    The remaining two cases are handled similarly to the first two,
    with the only difference that the function $\mathcal{X}_{1}(A, 8)$ is replaced by 
    $\mathcal{X}_{2}(A, 4)$.
    Thus, in the case $\rho_1 = 1$, the following estimate holds:
    \begin{equation}\label{eq_lambda_k_le_2}
        \#\{k \in \mathbb{N} \mid \lambda_k \le L\}
        \ge 2 \sum_{\substack{t_2 = 1 \\ t_2 -\text{ odd}}}^{L}
        \mathcal{X}_{2}\left(\frac{L - t_2}{4} - 1, \, 4 \right).
    \end{equation}

        Applying inequalities \eqref{eq_neravenstvo_XLa} and \eqref{eq_neravenstvo_X_2_La} 
        to the estimates \eqref{eq_lambda_k_le} and \eqref{eq_lambda_k_le_2}, respectively,
        we obtain the following. In our case, $a = 4$ or $8$, so $a \le 8$, and the condition 
        $R \ge a + 1$ becomes $|t_2| \le L - 40$. Hence, for $L \ge 41$, we obtain the following 
        estimate, independent of $\rho_1$ and $\rho_2$:

        \begin{equation}\label{eq_eq}
            \# \{k \in \mathbb{N} \mid \lambda_k \le L\} \ge \sum_{\substack{t_2 = 1 \\ t_2 -\text{ odd}}}^{L-40} G\left(\frac{L-t_2}{4} - 1\right),
        \end{equation}
        where
        \[
            G(x) = \frac{x^2}{\pi^2} - \frac{9x(\log x + 3)}{8}.
        \]

        Using the estimates
        \[
            \int_{1}^{\frac{M+1}{2}} (N - 2x + 1)^k \, dx
            \le \sum_{\substack{t = 1 \\ t -\text{ odd}}}^{M} (N - t)^k
            \le \int_{0}^{\frac{M+1}{2}} (N - 2x + 1)^k \, dx
        \]
        and
        \[
            \sum_{\substack{t = 1 \\ t- \text{odd}}}^{M} (N - t) \log (N - t)
            \le \log(N) \sum_{\substack{t = 1 \\ t -\text{ odd}}}^{M} (N - t),
        \]
        we obtain that for any $L \ge 41$
        \[
            \#\{k \in \mathbb{N} \mid \lambda_k \le L\} \ge r(L) L^3,
        \]
        where the function $r(L)$ is defined by \eqref{eq_c_L}.

        It is clear that
        \[
            r(L) = \frac{1}{96 \pi^2} \left(1 + O\left(\frac{\log L}{L}\right)\right).
        \]
        Moreover, it is easy to check that for $L \ge 1000$, the function $r(L)$ is positive and strictly increasing.
        Hence, for any fixed $M \ge 1000$ and any $L \ge M$, we have
        \[
            \#\{k \in \mathbb{N} \mid \lambda_k \le L\} \ge r(M) L^3.
        \]
        We complete the proof by substituting $L = \sqrt[3]{\frac{n}{r(M)}}$ into this estimate.
        \end{proof}

    \begin{lemma}[On finding the maximal level] \label{le_o_poiske_max_urovnya}
    Let $0 < c < 1$ and $1 \le s$. Then for any $0 < \alpha < 3$, the following estimate holds:
    \begin{equation} \label{eq_max_uroven}
        \frac{(2n)!}{(n!)^3} \left(
        c \log \left(\left(\frac{c}{s}\right)^{n+1} L^3\right)
        \right)^{n}
        \gg \frac{L^h}{\log L},
    \end{equation}
    where
    \begin{equation}\label{eq_n_lnL}
        n = \left\lfloor \alpha \frac{\log L}{\log \frac{s}{c}} \right\rfloor - 1.
    \end{equation}
    
    $h = \frac{\alpha}{\log \frac{s}{c}} \log \left(\frac{4ec\log (\frac{s}{c})(3-\alpha) }{\alpha}\right)$.

    \end{lemma}
    \begin{proof}
    Denote the left-hand side of \eqref{eq_max_uroven} by $X$.
    From \eqref{eq_n_lnL} it follows that
    \[
        \left(\frac{c}{s}\right)^{n+1} L^3 \ge L^{3 - \alpha}.
    \]

    Using Stirling's formula, we obtain the asymptotic expansion
    \[
        \frac{(2n)!}{(n!)^3} \sim \frac{(4e)^n}{\pi \sqrt{2} \, n^{n+1}}.
    \]

    Then
    \[
        X = \frac{(2n)!}{(n!)^3} \left(c \log \left(\left(\frac{c}{s}\right)^{n+1} L^3\right)\right)^n
        \gg
        \frac{\left(4ec(3 - \alpha) \log L\right)^n}{n^{n+1}}
        \gg \frac{1}{\log L} \left( \frac{4ec(3 - \alpha) \log L}{n} \right)^n.
    \]
    Applying \eqref{eq_n_lnL} twice, we obtain
    \[
        X \gg \frac{1}{\log L} \left( \frac{4ec \log \left( \frac{s}{c} \right) (3 - \alpha)}{\alpha} \right)^n
        \gg
        \frac{L^h}{\log L},
    \]
    where
    \[
        h = \frac{\alpha}{\log \left( \frac{s}{c} \right)} \log \left( \frac{4ec \log \left( \frac{s}{c} \right) (3 - \alpha)}{\alpha} \right).
    \]
    \end{proof}

 \begin{proof}[Proof of Theorem \ref{th_reshenie_combinatornoy_zadachi}.]
    Let $\mathcal{Q}$ denote the subset of $(\mathbb{Z}/2\mathbb{Z})^{n}$
    consisting of all $\mathbf{q} = (q_1, \ldots, q_n)$ such that among the $q_i \in \mathbb{Z}/2\mathbb{Z}$
    there are exactly $n$ ones and $n$ zeros.
    Denote by $\mathcal{Z}_{n, \mathbf{q}}(L)$ the number of tuples counted by $\mathcal{Z}_n(L)$
    with the additional restriction that the residue tuple $(q_1, \ldots, q_n)$ is equal to $\mathbf{q}$. Then $\mathcal{Z}_n(L)$ decomposes as the sum
    \[
        \mathcal{Z}_n(L) = \sum_{\mathbf{q} \in \mathcal{Q}} \mathcal{Z}_{n, \mathbf{q}}(L).
    \]
    Define the column of residues $\mathbf{r} = (r_1, \ldots, r_n)$ by
                \begin{equation*}
                \begin{pmatrix}
                    &q_1&\\
                    &q_2&\\
                    &q_3&\\
                    & \vdots &\\
                    &q_{2n}&\\
                \end{pmatrix} \equiv 
                \begin{pmatrix}
                    &1 & 0 & 0 & \ldots & 0\\
                    &1 & 1 & 0 & \ldots & 0\\
                    &1 & 1 & 1 & \ldots & 0\\
                    & \cdots & \cdots & \cdots & \ddots  & 0\\
                    &1 & 1 & 1 & \ldots & 1\\
                \end{pmatrix}
                \begin{pmatrix}
                    &r_1 + 1&\\
                    &r_2 + 1&\\
                    &r_3 + 1&\\
                    & \vdots &\\
                    &r_{2n} + 1&\\
                \end{pmatrix}
                \mod 2.
            \end{equation*} 
    Then $\mathcal{Z}_{n, \mathbf{q}}(L)$ counts the number of tuples $(m_1, t_1, \ldots, m_{2n}, t_{2n})$ such that
    \[
        \frac{m_j}{2} \equiv r_j \pmod 2.
    \]
    We estimate $\mathcal{Z}_{n, \mathbf{q}}(L)$ for a fixed $\mathbf{q}$.
    Note that the following inequality holds:
    \begin{gather} \label{ga_neravenstvo_norma}
        \left\| \prod_{j=1}^{2n} A(m_j, t_j) \right\|
        \le \prod_{i=1}^{n} \| A(m_{2i-1}, t_{2i-1}) A(m_{2i}, t_{2i}) \|.
    \end{gather}
    It is easy to see that
    \[
        \| A(m_{2i-1}, t_{2i-1}) A(m_{2i}, t_{2i}) \|
        = m_{2i}(2m_{2i-1} + |t_{2i-1}| + 1) + |t_{2i}|.
    \]
    Define the numbers $\lambda_k(\rho_1, \rho_2)$ as in Lemma \ref{le_assimptotika_mu_i}.
    From the definition of the sequence $\lambda_k(\rho_1, \rho_2)$ and inequality \eqref{ga_neravenstvo_norma}, it follows that
    \[
        \mathcal{Z}_{n, \mathbf{q}}(L) \ge \#\{(k_1, \ldots, k_n) \mid \lambda_{k_1} \cdots \lambda_{k_n} \le L\}.
    \]

    From Lemma \ref{le_assimptotika_mu_i}, for any $M \ge 1000$, we obtain the estimate
    \begin{gather}\label{eq_Z_n_q_estimate}
        \mathcal{Z}_{n, \mathbf{q}}(L) \ge 
        \left\{ (k_1, \ldots, k_n) \in \mathbb{N}^n \;\middle|\; k_i \ge s, \;
        \sqrt[3]{\frac{k_1}{c}} \cdots \sqrt[3]{\frac{k_n}{c}} \le L \right\},
    \end{gather}
    where $c = r(M)$, $s = M^3 r(M)$, and the function $r(M)$ is defined by \eqref{eq_c_L}.

    Since the right-hand side is independent of $\mathbf{q}$, we obtain the estimate
    \[
        \mathcal{Z}_n(L) \ge \binom{2n}{n}
        \left\{ (k_1, \ldots, k_n) \in \mathbb{N}^n \;\middle|\; k_i \ge s, \;
        \sqrt[3]{\frac{k_1}{c}} \cdots \sqrt[3]{\frac{k_n}{c}} \le L \right\}.
    \]

    We now estimate the counting function
    \[
        \#\{(k_1, \ldots, k_n) \in \mathbb{N}^n \mid k_i \ge s, \; k_1 \cdots k_n \le c^n L^3\}
        \ge
        \int\limits_{\substack{x_1 \cdots x_n \le c^n L^3 \\ s \le x_i}} dx_1 \cdots dx_n.
    \]

    Using standard methods of multidimensional integration, one obtains the following formula:
    \[
        \int\limits_{\substack{x_1 \cdots x_n \le A \\ s \le x_i}} dx_1 \cdots dx_n
        = \frac{s^n}{(n-1)!} \int_{1}^{A / s^n} (\log y)^{n-1} \, dy.
    \]
    Thus, we obtain the estimate
    \[
        \sum_{n=1}^{\infty} \mathcal{Z}_n(L) \ge \sum_{n=1}^{\infty} Q_n,
    \]
    where
    \[
        Q_n = \frac{(2n)! \, s^n}{(n!)^2 (n-1)!} \int_{1}^{\left(\frac{c}{s}\right)^n L^3} (\log y)^{n-1} \, dy.
    \]
    Using integration by parts, one easily derives the identity
    \[
        \frac{1}{(n-1)!} \int_{1}^{A} (\log y)^{n-1} \, dy
        + \frac{1}{n!} \int_{1}^{A} (\log y)^n \, dy
        = \frac{A (\log A)^n}{n!}.
    \]
    Using this identity, one readily obtains the estimate
    \[
        Q_n + Q_{n+1} \ge
        \frac{(2n)!}{(n!)^3} \left(
        \frac{c^{n+1}}{s} L^3 \log^n \left( \left( \frac{c}{s} \right)^{n+1} L^3 \right)
        \right).
    \]
    From Lemma \ref{le_o_poiske_max_urovnya}, for any $0 < \alpha < 3$, it follows that
    \[
        \sum_{n=1}^{\infty} \mathcal{Z}_n(L) \gg \frac{L^{3 + h}}{\log L},
    \]
    where
    \[
        h = \frac{\alpha}{\log \left( \frac{s}{c} \right)}
        \log \left( \frac{4ec \log \left( \frac{s}{c} \right) (3 - \alpha)}{\alpha} \right).
    \]
    Numerical computations show that the maximum of $h$ is attained approximately 
    at \[M = e^{11.2039} \text{ and } \alpha = 0.330679.\]
    Substituting these precise values of $M$ and $\alpha$ yields 
    \[h = 0.011057\ldots\]
    Therefore, the estimate with this $h$ is certainly valid.

    \end{proof}

    \begin{lemma}\label{le_neravenstvo_kompleks_cyklov}
        Consider a set of positive real numbers $\ell_0$ and $\ell_{(m_1, t_1, \ldots, m_n, t_n)}$, where $n \in \mathbb{N}$, $m_i \ge 0$, and $t_i \in \mathbb{R}$,
        satisfying the inequalities
        \begin{align}
            \ell_{(m_1, t_1, m_2, t_2)} &\le m_2 \ell_{(m_1, t_1)} + |t_2| \ell_0 + m_2 C, \label{al_neravenstvo_kompleks_cyklov} \\
            \ell_{(m_1, t_1, \ldots, m_n, t_n)} &\le m_n \ell_{(m_1, t_1, \ldots, m_{n-1}, t_{n-1})}
            + |t_n| \ell_{(m_1, t_1, \ldots, m_{n-2}, t_{n-2})}
            + m_n C \quad \text{for } n \ge 3, \label{al_neravenstvo_kompleks_cyklov_1}
        \end{align}
        where the constant $C > 0$ is independent of $n$.

        Then for any $n \ge 2$,
        \begin{gather*}
            \ell_{(m_1, t_1, \ldots, m_n, t_n)} \le
            P_n \cdot \ell_{(m_1, t_1)}
            + Q_n \cdot \ell_0
            + R_n \cdot C,
        \end{gather*}
        where $P_n, Q_n, R_n$ are defined by
        \begin{equation} \label{eq_A_1_A_2_A_n}
            \begin{pmatrix}
                P_n \\
                Q_n \\
                R_n
            \end{pmatrix}
            =
            A(m_2, t_2)
            \cdots
            A(m_n, t_n)
            \begin{pmatrix}
                1 \\
                0 \\
                0
            \end{pmatrix}.
        \end{equation}

    \end{lemma}
    \begin{proof} 
    The sequences $\{P_n\}, \{Q_n\}, \{R_n\}$ satisfy the recurrence relation
    \begin{equation} \label{eq_P_Q_R}
        \begin{pmatrix}
            P_{n+1} \\
            Q_{n+1} \\
            R_{n+1}
        \end{pmatrix}
        =
        m_{n+1}
        \begin{pmatrix}
            P_n \\
            Q_n \\
            R_n
        \end{pmatrix}
        +
        |t_{n+1}|
        \begin{pmatrix}
            P_{n-1} \\
            Q_{n-1} \\
            R_{n-1}
        \end{pmatrix}
        +
        m_{n+1}
        \begin{pmatrix}
            0 \\
            0 \\
            1
        \end{pmatrix}.
    \end{equation}
    For $n = 2$, the statement of the lemma reduces to \eqref{al_neravenstvo_kompleks_cyklov}.  
    We proceed by induction from $n-1$ to $n$. Applying the induction hypothesis to inequality \eqref{al_neravenstvo_kompleks_cyklov_1}, we obtain
    \begin{gather*}
        \ell_{(m_1, t_1, \ldots, m_n, t_n)} \le \\
        m_n \left( P_{n-1} \cdot \ell_{(m_1, t_1)}
        + Q_{n-1} \cdot \ell_0
        + R_{n-1} \cdot C \right)
        + |t_n| \left( P_{n-2} \cdot \ell_{(m_1, t_1)}
        + Q_{n-2} \cdot \ell_0
        + R_{n-2} \cdot C \right)
        + m_n C.
    \end{gather*}
    Using \eqref{eq_P_Q_R}, we obtain the desired inequality.
    \end{proof}

    \subsection{Proof of Theorem \ref{th_ocenka_S_1_2}.}

    Let $y \in H_1(S_{1,2}; \mathbb{Z}) / \langle [\delta] \rangle$ be a primitive homology class 
    such that $y = x + \mathbb{Z}[\delta]$, where $\delta$ is a geodesic 
    representing a loop around a puncture of $S_{1,2}$. We construct the cycle complex $\mathcal{B}_{1,2}$ 
    for the class $y$ and identify the geodesics 
    belonging to the class $y$ with vertices of $\mathcal{B}_{1,2}$. 
    From Theorem \ref{th_o_topologii_kompleksa_ciklov}, it follows that $\mathcal{B}_{1,2}$ is a tree; 
    moreover, from the definition 
    of $\mathcal{B}_{1,2}$, it follows that geodesics corresponding to adjacent vertices are disjoint.
    Abusing terminology slightly, we shall call the length of a vertex the length of the corresponding 
    geodesic on $S_{1,2}$.

    We enumerate the vertices of $\mathcal{B}_{1,2}$ as follows. 
    Let $\alpha_0$ be a vertex of minimal length in $\mathcal{B}_{1,2}$, and let $\alpha_{(0,1)}$ be a shortest vertex adjacent to 
    $\alpha_0$ in $\mathcal{B}_{1,2}$. If the choice of $\alpha_0$ and $\alpha_{(0,1)}$ is not unique, we fix any pair satisfying the above conditions.
    Let $\Sigma_{\alpha}$ denote the surface obtained by cutting 
    $S_{1,2}$ along the geodesic $\alpha$. Note that $\alpha_0$ is non-separating;
    hence, $\Sigma_{\alpha_0}$ is homeomorphic to $S_{0,2}^{2}$. For integral multicurves 
    on $\Sigma_{\alpha_0}$, we fix the special Dehn--Thurston coordinates 
    whose construction and properties are described in Lemma \ref{le_specialnye_coordinaty_D_T}, 
    taking $\gamma_0 = \alpha_{(0,1)}$ and $\gamma_1 = \alpha_0$. 
    Note that in this case item 2 of Lemma \ref{le_specialnye_coordinaty_D_T} does not apply, since 
    we take as $\mathcal{P} = \{\gamma_0\}$ the pants decomposition corresponding to the longer, rather than the shorter, 
    of the two curves $\alpha_0$ and $\alpha_{(0,1)}$.

    From Lemma \ref{le_specialnye_coordinaty_D_T}, it follows that the curves in $\mathcal{B}_{1,2}$ adjacent to $\alpha_0$
    are encoded by the following two sets of indices:
    \begin{enumerate}
        \item $\left\{ (m_1, t_1) \in 2\mathbb{N} \times \mathbb{Z} \mid \gcd\left(\frac{m_1}{2}, t_1\right) = 1,\ t_1 -\text{ odd},\ \frac{m_1}{2} -\text{ even} \right\} \cup \{(0, 1)\}$;
        \item $\left\{ (m_1, t_1) \in 2\mathbb{N} \times \mathbb{Z} \mid \gcd\left(\frac{m_1}{2}, t_1\right) = 1,\ t_1 -\text{ odd},\ \frac{m_1}{2} -\text{ odd} \right\}$.
    \end{enumerate}
    By item 5 of Lemma \ref{le_specialnye_coordinaty_D_T}, this partition corresponds to the two 
    possible homology classes $\left[\alpha_{(0, 1)}\right] \pm [\delta]$.
    The case
    \[
        \left\{ (m_1, t_1) \in 2\mathbb{N} \times \mathbb{Z} \;\middle|\; \gcd\left(\frac{m_1}{2}, t_1\right) = 1,\ t_1 -\text{ even} \right\}
    \]
    is not taken into account, since it contains geodesics that separate the surface $S_{1,2}$. Consequently,
    all curves with even $t_1$ are homologically trivial and do not correspond to any vertex of the complex $\mathcal{B}_{1,2}$.
    Denote by $\alpha_{(m_1, t_1)}$ the geodesic corresponding to a point of $\mathcal{B}_{1,2}$ and
    having coordinates $(m_1, t_1)$. Thus we have introduced a 
    numbering of the vertices of $\mathcal{B}_{1,2}$ that are at distance $1$ from $\alpha_0$.

    \begin{remark}
    Note that here (and below), for all curves corresponding to vertices of the cycle complex $\mathcal{B}_{1,2}$,
    we always choose the orientation such that the homology classes of these curves in the group  
    $H_1(S_{1,2}; \mathbb{Z}) / \langle [\delta] \rangle$ are equal to $x + \mathbb{Z}[\delta]$. Hence, for each of 
    these curves, a specific $\sigma(\gamma)$ arises, which appears in formulas \eqref{eq_gomology_formula}.
    \end{remark}
    We now introduce a numbering for the vertices at distance $2$ from $\alpha_0$. Each such vertex 
    is adjacent to some vertex $\alpha_{(m_1, t_1)}$ and hence lies in the surface $\Sigma_{\alpha_{(m_1,t_1)}}$.
    Moreover, $\alpha_0 \subset \Sigma_{\alpha_{(m_1,t_1)}}$ for any curve $\alpha_{(m_1, t_1)}$ in $\mathcal{B}_{1,2}$.
    For each surface $\Sigma_{\alpha_{(m_1,t_1)}}$, we fix the special Dehn--Thurston coordinates (see Lemma \ref{le_specialnye_coordinaty_D_T}), 
    taking $\gamma_0 = \alpha_0$ and $\gamma_1 = \alpha_{(m_1,t_1)}$.
    For these coordinates, inequality \eqref{al_specialnye_metricheskoe_neravenstvo} holds, since $\alpha_0$ is a vertex of minimal length in $\mathcal{B}_{1,2}$.
    From Lemma \ref{le_specialnye_coordinaty_D_T}, it follows that the curves in $\mathcal{B}_{1,2}$ adjacent to $\alpha_{(m_1, t_1)}$
    are encoded by two sets of indices:
    \begin{enumerate}
    \item $\left\{ (m_2, t_2) \in 2\mathbb{N} \times \mathbb{Z} \mid \gcd\left(\frac{m_2}{2}, t_2\right) = 1,\ t_2 -\text{ odd},\ \frac{m_2}{2} -\text{ even} \right\} \cup \{(0, 1)\}$;
    \item $\left\{ (m_2, t_2) \in 2\mathbb{N} \times \mathbb{Z} \mid \gcd\left(\frac{m_2}{2}, t_2\right) = 1,\ t_2 -\text{ odd},\ \frac{m_2}{2} -\text{ odd} \right\}$.
    \end{enumerate}
    The case of even $t_2$ again corresponds to separating geodesics, so we only consider odd $t_2$.
    Denote by $\alpha_{(m_1, t_1, m_2, t_2)}$ the geodesics belonging to the surface $\Sigma_{\alpha_{(m_1, t_1)}}$ and having coordinates $(m_2, t_2)$.
    Note that the geodesics $\alpha_{(m_1, t_1, 0, 1)}$ coincide with $\alpha_0$, so we do not include them.
    Since $\mathcal{B}_{1,2}$ has no cycles, there are no other index sets corresponding to the same vertex.  
    Thus, the vertices at distance $2$ from $\alpha_0$ are in bijection with the following set of indices: 
    \[\left\{(m_1, t_1, m_2, t_2) \in \left(2 \mathbb{N} \times \mathbb{Z} \cup \{(0, 1)\} \right)\times 2 \mathbb{N} \times \mathbb{Z} \mid 
    \gcd\left(\frac{m_i}{2}, t_i\right) = 1, \ t_i -\text{ odd}\right\}.\]    Let $\ell_{(m_1, t_1, \ldots, m_n, t_n)}$ denote the length of $\alpha_{(m_1, t_1, \ldots, m_n, t_n)}$, and let 
    $\ell_0$ denote the length of $\alpha_0$.
    Then from Lemma \ref{le_vypuklost_B} we obtain the inequality
    \[
        \ell_{(m_1, t_1)} \le \frac{\ell_0 + \ell_{(m_1, t_1, m_2, t_2)}}{2}.
    \]
    Since $\alpha_0$ is shortest, we have $\ell_0 \le \ell_{(m_1, t_1, m_2, t_2)}$.
    Combining this with the previous inequality yields
    \[
        \ell_{(m_1, t_1)} \le \ell_{(m_1, t_1, m_2, t_2)}.
    \]

    Similarly to the vertices at distance $2$ from $\alpha_0$, we describe all remaining vertices of $\mathcal{B}_{1,2}$.
    Suppose that for vertices at distance $n \ge 2$ from $\alpha_0$, a numbering $\alpha_{(m_1, t_1, \ldots, m_n, t_n)}$ has been fixed, and the inequality
    \[
        \ell_{(m_1, t_1, \ldots, m_{n-1}, t_{n-1})} \le \ell_{(m_1, t_1, \ldots, m_n, t_n)}
    \]
    holds.
    Then each geodesic adjacent to $\alpha_{(m_1, t_1, \ldots, m_n, t_n)}$ lies on the surface 
    $\Sigma_{\alpha_{(m_1, t_1, \ldots, m_n, t_n)}}$. Since $\alpha_{(m_1, t_1, \ldots, m_n, t_n)}$ is 
    non-separating on $S_{1,2}$, the surface $\Sigma_{\alpha_{(m_1, t_1, \ldots, m_n, t_n)}}$ is again 
    homeomorphic to $S_{0,2}^{2}$. For $\Sigma_{\alpha_{(m_1, t_1, \ldots, m_n, t_n)}}$, we fix
    the special Dehn--Thurston coordinates, taking 
    \[\gamma_0 = \alpha_{(m_1, t_1, \ldots, m_{n-1}, t_{n-1})} 
    \text{ and } \gamma_1 = \alpha_{(m_1, t_1, \ldots, m_n, t_n)}.\]
    In view of the inequality
    \[
        \ell_{(m_1, t_1, \ldots, m_{n-1}, t_{n-1})} \le \ell_{(m_1, t_1, \ldots, m_n, t_n)},
    \]
    for integral multicurves on the surface $\Sigma_{\alpha_{(m_1, t_1, \ldots, m_n, t_n)}}$,
    the inequality \eqref{al_specialnye_metricheskoe_neravenstvo} holds.
    As before, we denote the vertices at distance $n+1$ from $\alpha_0$ by
    $\alpha_{(m_1, t_1, \ldots, m_{n+1}, t_{n+1})}$, where $m_{n+1}$ is even and positive,
    $t_{n+1}$ is odd and integral, and $\gcd\left(\frac{m_{n+1}}{2}, t_{n+1}\right) = 1$.
    Then from Lemma \ref{le_vypuklost_B} and the inequality
    $\ell_{(m_1, t_1, \ldots, m_{n-1}, t_{n-1})} \le \ell_{(m_1, t_1, \ldots, m_n, t_n)}$,
    it follows that
    \[
        \ell_{(m_1, t_1, \ldots, m_n, t_n)} \le \ell_{(m_1, t_1, \ldots, m_{n+1}, t_{n+1})}.
    \]
    Since $\mathcal{B}_{1,2}$ has no cycles and the case $m_{n+1} = 0$ is excluded, the vertices
    of $\mathcal{B}_{1,2}$ at distance $n$ from $\alpha_0$ are in
    bijection with the following set of indices $(m_1, t_1, \ldots, m_n, t_n)$:
    \begin{enumerate}
        \item $(m_1, t_1) \in 2\mathbb{N} \times (2\mathbb{Z} + 1) \cup \{(0, 1)\}$;
        \item $(m_i, t_i) \in 2\mathbb{N} \times (2\mathbb{Z} + 1)$ for $i \ge 2$;
        \item $\gcd\left(\frac{m_i}{2}, t_i\right) = 1$.
    \end{enumerate}
    At any step except the first, the inequality $\ell_{\gamma_0} \le \ell_{\gamma_1}$ holds,
    so formula \eqref{al_specialnye_metricheskoe_neravenstvo} applies.
    Hence, for $n = 2$, we have
    \[
        \ell_{(m_1, t_1, m_2, t_2)} \le m_2 \ell_{(m_1, t_1)} + |t_1| \ell_0 + m_2 C;
    \]
    and for any $n \ge 3$,
    \[
        \ell_{(m_1, t_1, \ldots, m_n, t_n)} \le m_n \ell_{(m_1, t_1, \ldots, m_{n-1}, t_{n-1})}
        + |t_n| \ell_{(m_1, t_1, \ldots, m_{n-2}, t_{n-2})}
        + m_n C,
    \]
    where the constant $C$ is independent of $n$.

    From Lemma \ref{le_neravenstvo_kompleks_cyklov}, we obtain the inequalities
    \[
        \ell_{(m_1, t_1, \ldots, m_n, t_n)} \le
        P_n \cdot \ell_{(m_1, t_1)}
        + Q_n \cdot \ell_0
        + R_n \cdot C,
    \]
    where $P_n, Q_n, R_n$ are defined by
    \begin{equation*}
        \begin{pmatrix}
            P_n \\
            Q_n \\
            R_n
        \end{pmatrix}
        =
        A(m_2, t_2) \cdots A(m_n, t_n)
        \begin{pmatrix}
            1 \\
            0 \\
            0
        \end{pmatrix}.
    \end{equation*}

    \[
        \ell_{(m_1, t_1, \ldots, m_n, t_n)} \le \max\{C, \ell_0, \ell_{(m_1, t_1)}\} \left( P_n + Q_n + R_n \right).
    \]
    Let $A_n$ denote the matrix $A(m_n, t_n)$, and let $e_1$ be the vector $(1, 0, 0)^t$.
    Then
    \begin{align*}
        \max\{C, \ell_0, \ell_{(m_1, t_1)}\} (P_n + Q_n + R_n)
        &= \max\{C, \ell_0, \ell_{(m_1, t_1)}\} \| A_2 \cdots A_n e_1 \| \le\\
         \le \max\{C, \ell_0, \ell_{(m_1, t_1)}\} \| A_2 \cdots A_n \| \| e_1 \| 
        &= \max\{C, \ell_0, \ell_{(m_1, t_1)}\} \| A_2 \cdots A_n \|,
    \end{align*}
    where we use the $L^1$-norm of a vector, $\|v\| = |v_1| + |v_2| + |v_3|$, and the corresponding operator norm
    \[
        \|A\| = \max_{1 \le j \le 3} \sum_{i=1}^{3} |a_{ij}|.
    \]
    Thus, we obtain the following inequality:
    \begin{align} \label{al_ocenka_dliny_v_B_1_2}
        \ell_{(m_1, t_1, \ldots, m_n, t_n)} \le \max\{C, \ell_0, \ell_{(m_1, t_1)}\} \| A_2 \cdots A_n \|.
    \end{align}

    It remains to determine the restriction on the parameters that guarantees the equality
    $\left[\alpha_{(m_1, t_1, \ldots, m_n, t_n)}\right] = x$.
    Since the numbering was constructed using special Dehn--Thurston coordinates, from item 5 of
     Lemma \ref{le_specialnye_coordinaty_D_T} we have
    \begin{align*}
        \left[\alpha_{(m_1, t_1)}\right] &= \left[\alpha_0\right] + (-1)^{\frac{m_1}{2}} \sigma(\alpha_0) [\delta], \\
        \left[\alpha_{(m_1, t_1, \ldots, m_{n+1}, t_{n+1})}\right] &= \left[\alpha_{(m_1, t_1, \ldots, m_n, t_n)}\right] 
        + (-1)^{\frac{m_{n+1}}{2}} \sigma\left(\alpha_{(m_1, t_1, \ldots, m_n, t_n)}\right) [\delta],
    \end{align*}
    where the function $\sigma$ takes values in $\{-1, 1\}$. Moreover, it was noted in the construction of 
    the numbering that $\alpha_0 = \alpha_{(m_1, t_1, 0, 1)}$ and $\alpha_{(m_1, t_1, \ldots, m_n, t_n)} = 
    \alpha_{(m_1, t_1, \ldots, m_{n+1}, t_{n+1}, 0, 1)}$.
       
    Hence,
    \begin{align*}
        [\alpha_0] &= [\alpha_{(m_1, t_1)}] + \sigma(\alpha_{(m_1, t_1)})[\delta], \\
        [\alpha_{(m_1, t_1, \ldots, m_n, t_n)}] &=
        [\alpha_{(m_1, t_1, \ldots, m_{n+1}, t_{n+1})}] + \sigma(\alpha_{(m_1, t_1, \ldots, m_{n+1}, t_{n+1})})[\delta].
    \end{align*}
    From these and the previous equalities, one easily obtains
    \begin{align*}
        \sigma(\alpha_{(m_1, t_1, \ldots, m_n, t_n)}) &= \sigma(\alpha_0) \prod_{k=1}^n (-1)^{\frac{m_k}{2} + 1}, \\
        [\alpha_{(m_1, t_1, \ldots, m_n, t_n)}] &=
        [\alpha_0] - \sigma(\alpha_0) \left( \sum_{k=1}^{n} (-1)^{\sum_{j=1}^k \left( \frac{m_j}{2} + 1 \right)} \right) [\delta].
    \end{align*}

    Consider the case where $[\alpha_0] = x$. Let $q_k = \sum_{j=1}^k \left( \frac{m_j}{2} + 1 \right)$. Then for the equality
    \[
        [\alpha_0] = [\alpha_0] - \sigma(\alpha_0) \left( \sum_{k=1}^{n} (-1)^{q_k} \right) [\delta]
        = [\alpha_{(m_1, t_1, \ldots, m_n, t_n)}]
    \]
    to hold, it is necessary and sufficient that $n$ be even and that among the numbers $q_k$, exactly $\frac{n}{2}$ are even and $\frac{n}{2}$ are odd.
    By definition, the $q_k$ are related to the $m_i$ as follows:

    \begin{equation*}
        \begin{pmatrix}
            &q_1&\\
            &q_2&\\
            &q_3&\\
            & \vdots &\\
            &q_n&\\
            \end{pmatrix}=
            \begin{pmatrix}
            &1 & 0 & 0 & \ldots & 0\\
            &1 & 1 & 0 & \ldots & 0\\
            &1 & 1 & 1 & \ldots & 0\\
            & \cdots & \cdots & \cdots & \ddots  & 0\\
            &1 & 1 & 1 & \ldots & 1\\
        \end{pmatrix}
        \begin{pmatrix}
            &\frac{m_1}{2} + 1&\\
            &\frac{m_2}{2} + 1&\\
            &\frac{m_3}{2} + 1&\\
            & \vdots &\\
            &\frac{m_n}{2} + 1&\\
            \end{pmatrix}
    \end{equation*}
    Note that the matrix in the relation above is invertible over $\mathbb{Z}/2\mathbb{Z}$; hence, the row
    $\left( \frac{m_1}{2}, \ldots, \frac{m_n}{2} \right)$
    is in bijection with the row $(q_1, \ldots, q_n)$ up to parity.
    Therefore, for the equality
    $\left[\alpha_{(m_1, t_1, \ldots, m_n, t_n)}\right] = x$
    to hold, it is necessary and sufficient that
    $\alpha_{(m_1, t_1, \ldots, m_n, t_n)}$ correspond to a vertex in the complex $\mathcal{B}_{1,2}$,
    that $n$ be even, and that the parities of the row
    $\left( \frac{m_1}{2}, \ldots, \frac{m_n}{2} \right)$
    correspond to the parities of some row $(q_1, \ldots, q_n)$ in which exactly half of the entries are
    even and half are odd.
    
    From the above and estimate (\ref{al_ocenka_dliny_v_B_1_2}), it follows that
    \[
        h_{1,2}(L, x) \ge \sum_{(m_1, t_1)} \sum_{n=1}^{\infty} \mathcal{Z}_n\left( \frac{L}{\max\{C, \ell_0, \ell_{(m_1, t_1)}\}} \right)
        \ge \sum_{n=1}^{\infty} \mathcal{Z}_n\left( \frac{L}{\max\{C, \ell_0, \ell_{(0, 1)}\}} \right),
    \]
    where $\mathcal{Z}_n$ is defined in Theorem \ref{th_reshenie_combinatornoy_zadachi}.
    We complete the proof in the case $[\alpha_0] = x$ by applying Theorem \ref{th_reshenie_combinatornoy_zadachi}.
    The general case reduces to the case $[\alpha_0] = x$ by means of Theorem \ref{th_ne_zavisimosti_assimptotik_ot_klassa_gomologiy}.

\end{document}